\documentclass[final]{siamltex}

\usepackage{graphicx}
\usepackage{amstext}
\usepackage{amsmath}
\usepackage{amssymb}
\usepackage{amsfonts}

\usepackage{booktabs}
\usepackage{url}
\usepackage{import}
\usepackage{fancyvrb}
\usepackage{stmaryrd}
\usepackage{pdfsync}
\usepackage{graphicx}
\usepackage{algorithm}
\usepackage[numbers,square,sort&compress]{natbib}

\usepackage[shadow]{todonotes}

\newtheorem{lem}{Lemma}[section]
\newtheorem{prop}{Proposition}[section]
\newtheorem{thm}{Theorem}[section]
\newtheorem{rem}{Remark}[section]

\numberwithin{equation}{section}
\numberwithin{figure}{section}
\numberwithin{table}{section}

\newcommand{\R}{\mathbb{R}}
\newcommand{\N}{\mathbb{N}}
\newcommand{\foralls}{\forall\,}
\newcommand{\dx}{\,\mathrm{d}x}
\newcommand{\ds}{\,\mathrm{d}s}

\DeclareMathOperator{\Div}{div}

\DeclareMathOperator{\ord}{ord}
\DeclareMathOperator{\Kern}{ker}
\DeclareMathOperator{\Image}{im}
\DeclareMathOperator{\spann}{span}
\DeclareMathOperator{\diam}{diam}

\newcommand{\mesh}{\mathcal{T}}
\newcommand{\jump}[1]{[#1]}

\newcommand{\apriori}{\emph{a~priori}}

\DefineVerbatimEnvironment{code}{Verbatim}{frame=single,rulecolor=\color{blue}}

\newcommand{\bfu}{\boldsymbol{u}}
\newcommand{\bff}{\boldsymbol{f}}
\newcommand{\bfg}{\boldsymbol{g}}
\newcommand{\bfv}{\boldsymbol{v}}
\newcommand{\bfe}{\boldsymbol{e}}
\newcommand{\bfn}{\boldsymbol{n}}
\newcommand{\bfx}{\boldsymbol{x}}
\newcommand{\Oast}{\Omega^{\ast}}

\newcommand{\pO}{\Gamma}
\newcommand{\meshast}{\mathcal{T}^{\ast}}
\newcommand{\Fast}{\mathcal{F}^{\ast}_{\Gamma}}
\newcommand{\piast}{\pi^{\ast}_h}
\newcommand{\tnast}{\tn_{\ast}}

\newcommand{\mcV}{\mathcal{V}}
\newcommand{\mcE}{\mathcal{E}}

\newcommand{\mcC}{\mathcal{C}}
\newcommand{\mcA}{\mathcal{A}}
\newcommand{\tn}{|\mspace{-1mu}|\mspace{-1mu}|}
\newcommand{\Pzero}{P_h^{0,\mathrm{dc}}}
\newcommand{\Pone}{P_h^1}

\newcommand{\bfw}{\boldsymbol{w}}
\newcommand{\Wspace}{{\mathcal{V}_h}}

\newcommand{\nablan}{\partial_{\bfn}}

\title{\bf A stabilized Nitsche fictitious domain method for the Stokes problem}

\author{
  Andr\'e Massing\thanks{Simula Research Laboratory, Oslo, Norway}
  \and
  Mats G.\ Larson\thanks{Department of Mathematics, Ume{\aa} University, Ume{\aa}, Sweden.}
  \and Anders Logg\thanks{Simula Research Laboratory, Oslo, Norway}
  \and Marie E.\ Rognes\thanks{Simula Research Laboratory, Oslo, Norway}
}

\begin{document}

\maketitle

\begin{abstract}
  We develop a Nitsche fictitious domain method for the Stokes problem
  starting from a stabilized Galerkin finite element method with low
  order elements for both the velocity and the pressure.  By
  introducing additional penalty terms for the jumps in the normal
  velocity and pressure gradients in the vicinity of the boundary, we
  show that the method is inf-sup stable. As a consequence, optimal
  order \apriori{} error estimates are established. Moreover, the
  condition number of the resulting stiffness matrix is shown to be
  bounded independently of the location of the boundary. We discuss a
  general, flexible and freely available implementation of the method
  in three spatial dimensions and present numerical examples
  supporting the theoretical results.
\end{abstract}

\begin{keywords}
Fictitious domain, Stokes problem, stabilized finite element methods,
Nitsche's method
\end{keywords}

\begin{AMS}
65N12, 65N30, 65N85, 76D07
\end{AMS}

%------------------------------------------------------------------------------
\section{Introduction}

A frequently encountered problem in practical applications of the
finite element method is the generation of a high quality mesh
conforming to the computational domain. For instance, the simulation
of flow around an object embedded in a channel typically requires a
mesh discretizing the domain surrounding the object. If the domain is
complex, the mesh generation problem is highly
non-trivial. Furthermore, the mesh must be modified or regenerated
each time the object is translated, scaled or rotated, for example to
study the lift or drag for different angles of attack.

In fictitious domain finite element methods
\citep{GlowinskiKuznetsov2007,GlowinskiPanHeslaEtAl2001,yu2005dlm,JohanssonLarson},
the computational domain is instead represented by a, possibly
regular, background mesh and an interior surface; this situation is
illustrated in Figure~\ref{fig:stokes_fictitious_domain}. The mesh
generation problem is thus essentially avoided. However, new
challenges are introduced. The interior surface must be represented
and the intersection of the surface and the underlying mesh computed,
which is a complex task for three-dimensional domains.  Moreover, the
finite element formulation, and hence also its analysis and
implementation, involves elements of non-regular shapes induced by
this intersection.
\begin{figure}
  \begin{center}
    \includegraphics[width=0.5\textwidth]{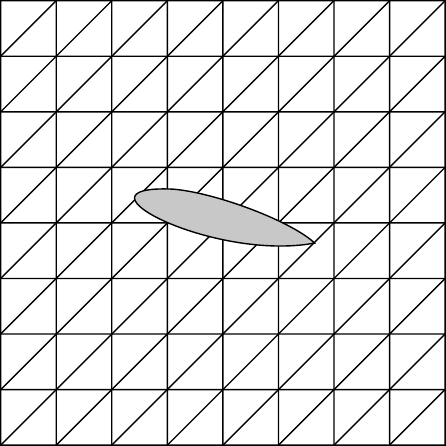}
    \caption{The stabilized Nitsche fictitious domain method presented
      in this work allows the simulation of Stokes flow around a
      possibly complex object (in this simplified illustration, a
      two-dimensional airfoil) embedded in a fixed background
      mesh. The object is defined by its boundary $\pO$, and the
      computational mesh (here the cut mesh surrounding the airfoil)
      is defined as the intersection of the fixed background mesh and
      the outside (or inside) of the boundary $\pO$.}
    \label{fig:stokes_fictitious_domain}
  \end{center}
\end{figure}

In this work, we consider a Nitsche fictitious domain method for the
Stokes problem: find the velocity
$\bfu: \Omega \subset \R^d \rightarrow
\R^d$ and the pressure $p:\Omega \to \R$ such that
\begin{subequations}\label{p2:eq:strongform}
  \begin{alignat}{3}
    - \Delta \bfu + \nabla p &= \bff & \quad & \text{in $\Omega$},
    \label{p2:eq:strong-stress}
    \\
    \nabla \cdot \bfu &=0 &\quad & \text{in $\Omega$},
    \label{p2:eq:strong-divergence}
    \\
    \bfu&=\bfg & \quad & \text{on $\pO$},
    \label{p2:eq:strong-dirichlet}
  \end{alignat}
\end{subequations}
where $\Omega$ denotes a bounded domain in $\R^d$, $d=2$ or $3$, with
Lipschitz boundary $\pO$, and where $\bff\in L^2(\Omega)$ is a given
body force and $\bfg \in H^{1/2}(\pO)$ is a prescribed boundary
velocity. To satisfy~\eqref{p2:eq:strong-divergence}, we assume that
$\int_{\pO} \bfn \cdot \bfg \ds = 0$ where $\bfn$ denotes the outward
pointing boundary normal. Moreover, we assume that $\int_{\Omega} p
\dx = 0$ to uniquely determine $p$.

The fictitious domain method introduced in this paper is based on a
least squares stabilized finite element method with low order finite
element spaces. In particular, we consider both the case of continuous
piecewise linear vector fields for the velocity and continuous
piecewise linears for the pressure, and the case of continuous
piecewise linear vector fields for the velocity and piecewise
constants for the pressure. We prove stability and optimal \apriori{}
error estimates as well as optimal estimates for the condition
number. These results rely on the introduction of stabilization terms
for the jump in the normal gradients at faces associated with elements
intersecting the boundary. Our method is closely related to a very recent
report of~\citet{BurmanHansbo2011a}, but the
analysis follows a different route. The present
work also differs from that of~\citet{BurmanHansbo2011a} in that our
methodology has been tested and implemented in three dimensions.
Similar results have been obtained for elliptic boundary problems
by~\citet{Burman2010,BurmanHansbo2011} and~\citet{JohanssonLarson}. In a related
work~\citep{Massing2012b}, we present a stabilized Nitsche overlapping
mesh method for the Stokes problem.

A central and unique contribution of the current work is the full and
general treatment of domains represented by arbitrary boundary
triangulations embedded in three-dimensional tetrahedral meshes. This
requires integration over arbitrary polyhedral domains resulting from
the subtraction of the embedded domain from the background mesh. The
intersection of the boundary and the background mesh is computed
efficiently using techniques from computational geometry.  The freely
available implementation is based on, but extends that of, our
previous work~\citep{Massing2012a}.

The remainder of this paper is organized as follows. In
Section~\ref{p2:sec:preliminaries}, we summarize the notation and
assumptions used throughout this work. The novel Nitsche fictitious
domain finite element formulation for the Stokes problem is then
introduced in Section~\ref{p2:sec:nitsche-fd-stokes}, while
Sections~\ref{p2:sec:approx-est}--\ref{p2:sec:apriori-estimate} are devoted
to its \apriori{} error analysis. We prove that the condition number
is bounded independently of the location of the boundary in
Section~\ref{p2:sec:condition-number}. A brief summary of key
implementation aspects is provided in Section~\ref{p2:sec:num-examples},
along with numerical investigations corroborating the theoretical
results and an example demonstrating the applicability of the
developed framework to complex 3D geometries.
Finally, we provide some concluding remarks in
Section~\ref{p2:sec:conclusion}.

%------------------------------------------------------------------------------
\section{Preliminaries}
\label{p2:sec:preliminaries}

The Nitsche fictitious domain finite element formulation involves
integration over various geometric entities. We here define these
entities and summarize the notation that will be used throughout this
paper for computational domains, meshes, function spaces and norms.

\subsection{Computational domain and meshes}
\label{p2:sec:notation-assumptions}

Let $\Omega$ be an open, bounded domain in $\R^d$ ($d = 2, 3$) with
Lipschitz boundary $\pO$.  We assume that $\Omega$ is a subset of a
larger polygonal domain $\Oast$; that is, $\Omega \subset \Oast$. We will
refer to $\Oast$ as the \emph{fictitious domain}. Let $\meshast$ be a
shape-regular tessellation of $\Oast$ such that $T \cap \Omega \neq
\emptyset$ for all $T \in \meshast$. The mesh $\meshast$ might be
constructed from a larger and easy-to-generate mesh
$\widehat{\meshast}$ by extracting a suitable submesh,
cf.~Figure~\ref{fig:computational-domain}. A facet $F$; that is, an
edge in two dimensions or a face in three dimensions, of the mesh
$\meshast$ is labeled an \emph{exterior facet} if it belongs to one
element only (and is thus a part of the boundary of $\Oast$) or an
\emph{interior facet} if it is shared by two elements. In the latter
case, we denote the two elements shared by the facet $F$ by $T^+_F$
and $T^-_F$. The set of all exterior facets defines the boundary mesh
$\partial_e \meshast$, while the set of all interior facets defines
the skeleton mesh $\partial_i \meshast$.
\begin{figure}
  \begin{center}
    \includegraphics[width=0.95\textwidth]{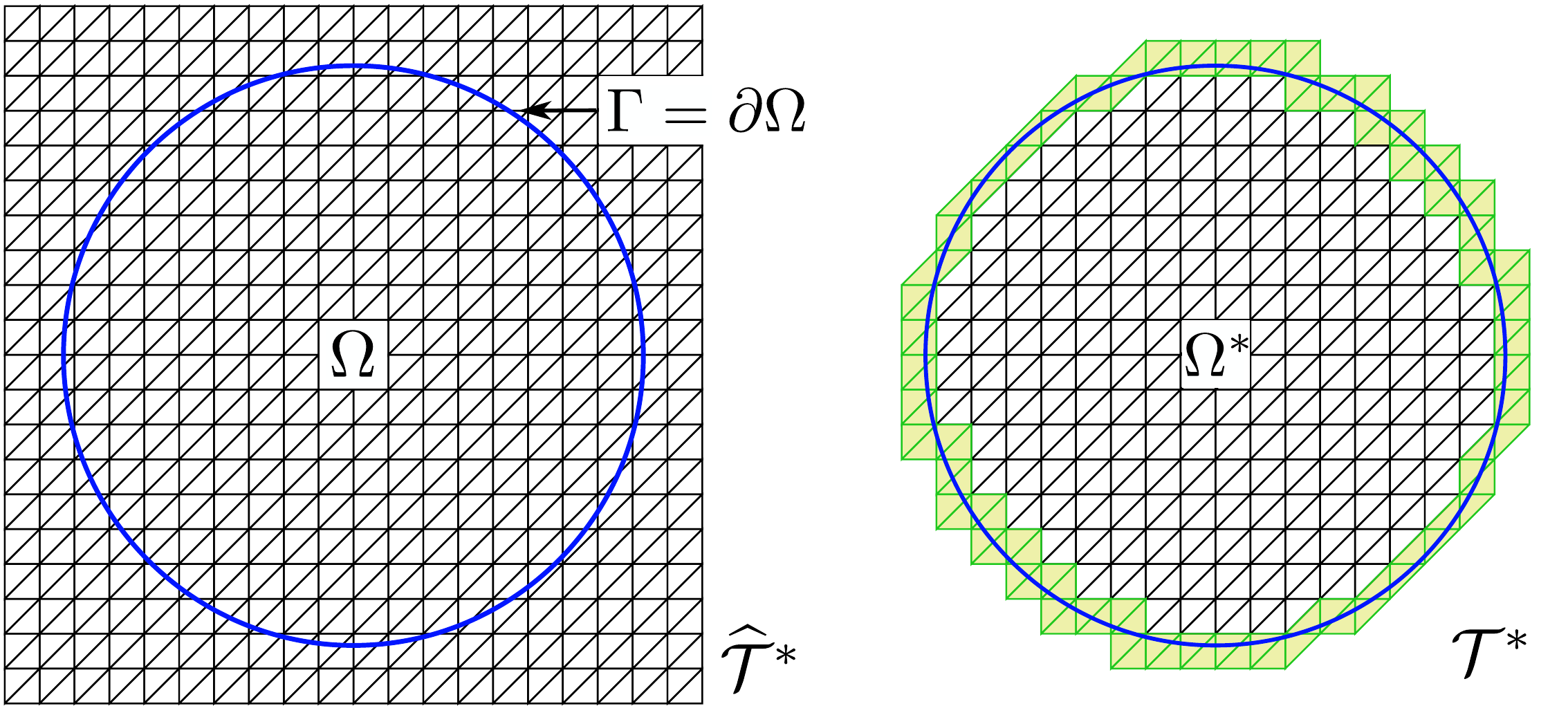}
    \caption{(Left) The computational domain $\Omega$ is defined as
      the inside or outside of a given boundary $\pO$ imposed on a
      fixed background mesh $\widehat{\meshast}$. (Right) The
      fictitious domain $\Oast$ is the union of the minimal subset
      $\meshast \subset \widehat{\meshast}$ covering $\Omega$.}
  \label{fig:computational-domain}
  \end{center}
\end{figure}

Given $\meshast$, we may define the \emph{cut mesh} $\mesh$ on
$\Omega$ as follows:
\begin{equation}
  \mesh = \{T \cap \overline{\Omega} : T \in \meshast\}.
\end{equation}
The corresponding boundary and skeleton meshes are defined accordingly
by $\partial_e \mesh = \{F \cap \overline{\Omega} : F \in \partial_e
\meshast\}$ and
$\partial_i \mesh = \{F \cap \overline{\Omega} : F \in \partial_i
\meshast\}$.
Note that $\mesh$, $\partial_e \mesh$ and $\partial_i
\mesh$ consist of both standard (simplicial) elements and facets, and
non-standard elements and facets. We will occasionally refer to the
former set as \emph{non-cut} elements or facets, and the latter set
as \emph{cut} elements or facets.

Next, let $\meshast_{\pO}$ be the subset of elements in $\meshast$
that intersect the boundary $\pO$:
\begin{equation}
  \label{p2:eq:define-cutting-cell-mesh}
  \meshast_{\pO} = \{T \in \meshast: T \cap \pO \neq \emptyset \}
\end{equation}
and introduce the notation $\Fast$ for the set of all interior facets
belonging to elements intersected by the boundary $\pO$:
\begin{equation}
  \label{p2:eq:define-cutting-facet-mesh}
  \Fast = \{ F \in \partial_i \meshast :\;
  T^+_F \cap \pO \neq  \emptyset
  \vee
  T^-_F \cap \pO \neq  \emptyset
  \}.
\end{equation}
Figure~\ref{fig:boundary-zone} illustrates this notation.

We assume that $\meshast$ and the boundary $\pO$ satisfy the
following geometric conditions:
\begin{itemize}
\item G1: The intersection between $\Gamma$ and a facet $F \in
  \partial_i \meshast$ is simply connected; that is, $\Gamma$ does not
  cross an interior facet multiple times.
\item G2: For each element $T$ intersected by
  $\Gamma$, there exists a plane $S_T$ and a piecewise smooth
  parametrization $\Phi: S_T \cap T \rightarrow \Gamma \cap T$.
\item G3: We assume that there is an integer $N>0$ such
  that for each element $T \in \meshast_{\pO}$ there exists an element
  $T' \in \meshast \setminus \meshast_{\pO}$ and at most $N$ elements
  $\{T\}_{i=1}^N$ such that $T_1 = T,\,T_N = T'$ and $T_i \cap
  T_{i+1} \in \partial_i \meshast,\; i = 1,\ldots N-1$.  In other
  words, the number of facets to be crossed in order to ``walk'' from
  a cut element $T$ to a non-cut element $T' \subset \Omega$ is
  bounded.
\end{itemize}
Similar assumptions were made by
\citet{HansboHansbo2002,BurmanHansbo2011} for the two dimensional case
and ensure that $\Gamma$ is reasonably resolved by $\meshast$.
\begin{figure}
  \begin{center}
    \includegraphics[width=0.95\textwidth]{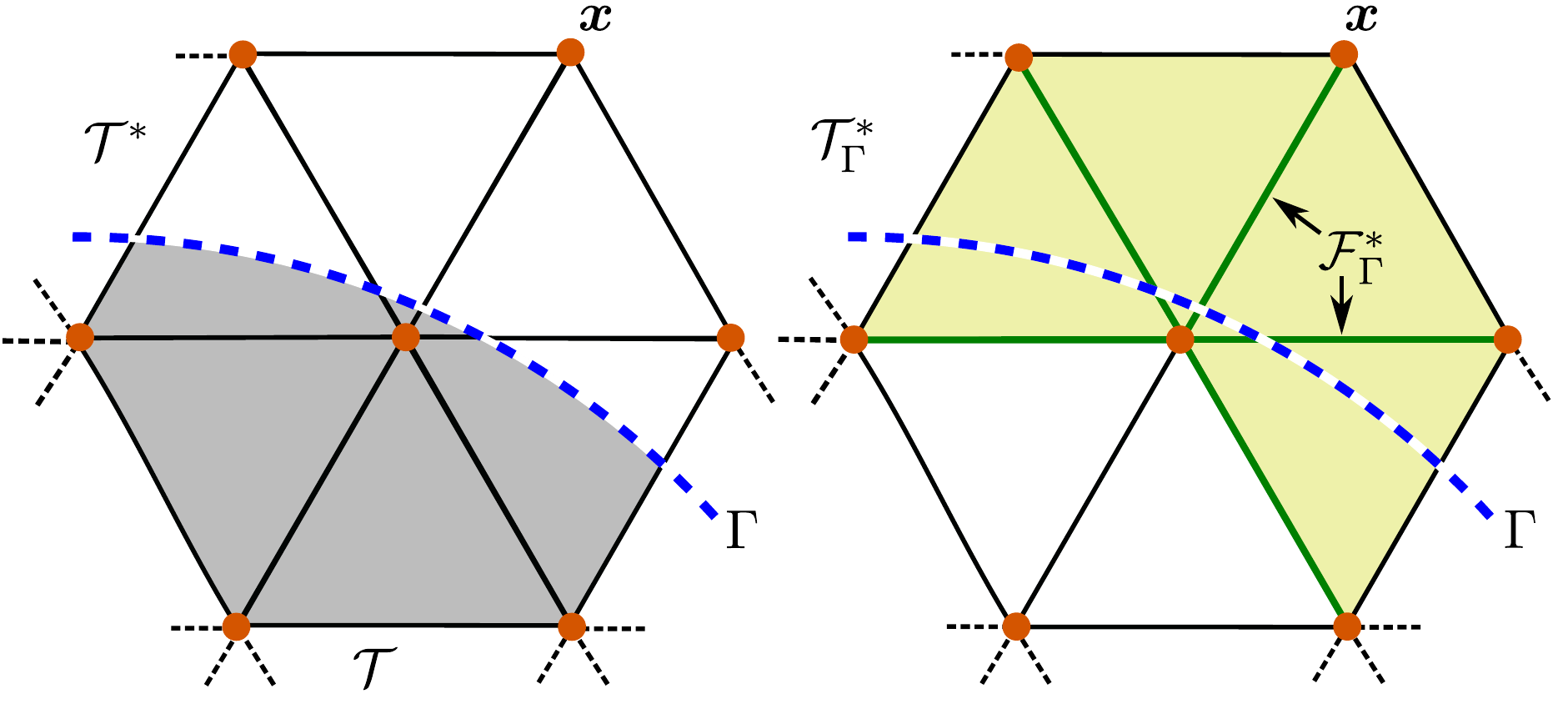}
    \caption{The boundary zone of the fictitious domain. (Left) The
      background mesh $\meshast$ together with the cut mesh
      $\mesh$. Observe that for the element associated with the node
      $\bfx$, only a small fraction resides inside the domain
      $\Omega$. (Right) The elements in yellow are intersected by the
      boundary and therefore part of the mesh
      $\meshast_{\pO}$. Interior facets of elements intersected by the
      boundary ($\Fast$) are marked in green.}
  \label{fig:boundary-zone}
  \end{center}
\end{figure}

\subsection{Finite element spaces}
\label{p2:ssec:fem-spaces}

We let the discrete velocity space $V_h$ be the space of continuous,
piecewise linear $\R^d$-valued vector fields defined relative to a
specified mesh, and let the pressure space $Q_h$ consist of either
piecewise constant or continuous piecewise linear elements, denoted by
$\Pzero$ and $\Pone$, respectively.

Here and below, let $\| \cdot \|_{s, \Omega}$ and $|\cdot|_{s,
  \Omega}$ denote the standard Sobolev norms and semi-norms on a
domain $\Omega$ for $s \in \N$. The corresponding inner products are
denoted by $(\cdot, \cdot)_{s, \Omega}$.  For $s = 0$, the subscript
$s$ is omitted.  Furthermore, we introduce the following
mesh-dependent norms for the velocity:
\begin{align}
  \tn \bfv \tn^2 &= \| \nabla \bfv \|^2_{\Omega}
  + \| h^{-1/2}  \bfv \|^2_{\pO}
  + \| h^{1/2} \bfn \cdot \nabla \bfv \|^2_{\pO},
  \\
  \label{p2:eq:triple-u-norm}
  \tn \bfv \tn_{\ast}^2 &= \| \nabla \bfv \|^2_{\Oast}
  + \| h^{-1/2}  \bfv \|^2_{\pO},
\end{align}
for the pressure:
\begin{align}
  \tn q \tn^2 &= \| q \|^2_{\Omega}
  + \| h^{1/2}  q \|^2_{\pO}, \\
  \label{p2:eq:triple-p-norm}
  \tn q \tn^2_{\ast} &= \| q \|_{\Oast}^2,
\end{align}
and for the product space:
\begin{align}
  \tn (\bfv,q) \tn^2 &= \tn \bfv \tn^2 + \tn q \tn^2,
  \\
  \label{p2:eq:triple-up-norm}
  \tn (\bfv,q) \tn^2_{\ast} &= \tn \bfv \tn^2_{\ast} +
  \tn q \tnast ^2.
\end{align}
Note that the $\tn \cdot \tn_{\ast}$-norms are defined on the
fictitious domain $\Oast$ and therefore represent proper norms for
the discrete finite element functions.  When mesh-dependent norms are
applied to non-finite element functions on a domain $\Omega$, we
always mean the evaluation of the norm on a tessellation $\mesh$ of
$\Omega$.

%------------------------------------------------------------------------------
\section{Finite element formulation}
\label{p2:sec:nitsche-fd-stokes}

Before we present the Nitsche fictitious domain method, we review a
pair of well-established stabilized finite element formulations for
the Stokes problem. These formulations are then extended to a
Nitsche-based fictitious domain method.

\subsection{Stabilized Stokes elements}
\label{p2:ssec:stabilized-and-bubble-elements}

Let $V_h$ and $Q_h$ be the velocity and pressure spaces introduced in
the previous section defined relative to a standard conforming
tessellation $\mesh$ of $\Omega$ and recall that $Q_h$ is defined to
be either $\Pone$ or $\Pzero$.  It is well-known that the mixed spaces
$V_h \times \Pone$ and $V_h \times \Pzero$ violate the inf-sup
condition for the ${[H^1_0(\Omega)]^d} \times L^2(\Omega)/\R$ variational formulation
of the Stokes problem~\eqref{p2:eq:strongform}, and thus, are not stable
in the Babu{\v{s}}ka--Brezzi sense~\citep{BrezziFortin1991}. Different
strategies can be employed to compensate for the lack of
stability~\citep{HughesFrancaEtAl1986,Hughes1989a,Kechkar1992,Bochev2006a},
whereof consistently stabilized methods are among the most
prominent~\citep{Franca1993,Barth2004}. Here, we consider consistently
stabilized discrete variational formulations of~\eqref{p2:eq:strongform},
with $\bfg = 0$, of the following form: find $(\bfu_h, p_h) \in V_h
\times Q_h$ such that
\begin{equation}
  A_h(\bfu_h, p_h ; \bfv_h, q_h) = L_h(\bfv_h,q_h)
  \quad \foralls (\bfv_h,q_h) \in V_h \times Q_h,
  \label{p2:eq:stabilized-stokes-form}
\end{equation}
where the bilinear and linear forms $A_h$ and $L_h$ are defined by
\begin{align}
  A_h(\bfu_h,p_h;\bfv_h,q_h)
  &= a_h(\bfu_h,\bfv_h) + b_h(\bfu_h,q_h) +
  b_h(\bfv_h,p_h) -
  c_h(\bfu_h, p_h; q_h),
  \nonumber %  \label{p2:eq:standard-stokes-A}
  \\
  L_h(\bfv_h,q_h) &= (\bff,\bfv_h) - \Phi_h(q_h).
  \label{p2:eq:standard-stokes-L}
\end{align}
Here, $a_h$ and $b_h$ are the standard forms
\begin{align}
  a_h(\bfu_h,\bfv_h) &= (\nabla \bfu_h, \nabla \bfv_h)_{\Omega},
  \label{p2:eq:standard-stokes-a}
  \\
  b_h(\bfv_h,p_h) &= -(\nabla \cdot \bfv_h, p_h)_{\Omega}.
  \label{p2:eq:standard-stokes-b}
\end{align}
The stabilization form $c_h$ is given by
\begin{align}
  c_h(\bfu_h, p_h; q_h) &=
  \begin{cases}
    \beta_0 \sum_{F \in \partial_i \mesh} \, h_F (\jump{p_h}, \jump{q_h})_{F}
    &\quad\text{if } Q_h = \Pzero, \\
    \beta_1 \sum_{T \in \mesh} \, h_T^2 (-\Delta \bfu_h + \nabla p_h, \nabla q_h)_{T}
    &\quad\text{if } Q_h = \Pone,\\
  \end{cases}
\end{align}
where $h_T$ denotes the diameter of element $T$, $h_F$ denotes the
average of the diameters of the elements sharing a facet $F$, $[v] =
v^+ - v^-$ is the jump in a function $v$ over each facet $F$:
$v^\pm(\bfx) = \lim_{t \rightarrow 0^+} v(\bfx \pm t \bfn)$ for $\bfx
\in F$, and $\beta_0$ and $\beta_1$ are positive stabilization
constants. In the case $Q_h = \Pone$, this stabilization is also known
as the pressure-Poisson stabilized Galerkin method. Note that $-\Delta
\bfu_h$ vanishes if $\bfu_h$ is piecewise linear and is only included
to clarify that the method is indeed consistent. We will therefore
simply write $c_h(p_h, q_h)$ when only finite element functions are
involved. The form $\Phi_h$ in~\eqref{p2:eq:standard-stokes-L} is, to
ensure consistency, defined to be
\begin{align}
  \Phi_h(q_h) &=
  \begin{cases}
    0
    &\quad\text{\hspace{3pt} if } Q_h = \Pzero, \\
    \beta_1 \sum_{T \in \mesh} h_T^2 (\bff, \nabla q_h)_T
    &\quad\text{\hspace{3pt} if } Q_h = \Pone.
  \end{cases}
\end{align}
Since $\jump{q_h} = 0$ for $q_h \in \Pone$ and $\nabla (q_h|_T) = 0$
for $q_h \in \Pzero$, we may express the two cases in a more compact
notation:
\begin{align}
  \label{p2:eq:stabilization_term_lhs}
  c_h(p_h, q_h) &=
  {\beta_0 \sum_{F \in \partial_i \mesh} h_F (\jump{p_h},
  \jump{q_h})_{F}}
  +
  { \beta_1 \sum_{T \in \mesh} h_T^2 (\nabla p_h, \nabla q_h)_{T}},
  \\
  \label{p2:eq:stabilization_term_rhs}
  \Phi_h(q_h) &= \beta_1 \sum_{T \in \mesh} h_T^2 (\bff, \nabla q_h)_T.
\end{align}

\subsection{A stabilized Nitsche fictitious domain method}
\label{p2:ssec:stokes-fd-fem}

Prior to stating the stabilized Nitsche fictitious domain formulation
for the Stokes problem, we introduce the following forms with
reference to the notation established in
Section~\ref{p2:sec:notation-assumptions}:
\begin{align}
  \label{p2:eq:nitsche-fd-form-a}
  a_h(\bfu_h,\bfv_h) &= (\nabla \bfu_h, \nabla \bfv_h )_{\Omega}
  - (\nablan \bfu_h, \bfv_h)_{\pO} -  (\nablan \bfv_h, \bfu_h)_{\pO}
   + \gamma  (h^{-1} \bfu_h, \bfv_h)_{\pO}, \\
  \label{p2:eq:nitsche-fd-form-b}
  b_h(\bfv_h, p_h) &= -(\nabla \cdot \bfv_h, p_h )_{\Omega}
  + (\bfn \cdot \bfv_h, p_h
  )_{\pO} ,
\end{align}
where $\nablan \bfv = \bfn \cdot \nabla \bfv$. Next, we
introduce the velocity ``ghost-penalty'' form:
\begin{align}
  \label{p2:eq:normal-derivative-jump-penalty}
  i_h(\bfu_h,\bfv_h) &= \beta_2 \sum_{F \in \Fast} h_F (
  [\nablan \bfu_h] , [ \nablan \bfv_h ])_F,
\end{align}
and the pressure ``ghost-penalty'' form:
\begin{align}
  j_h(p_h,q_h) &=
  \begin{cases}
    \beta_0 \sum_{F \in {\Fast}}
    h_F(\jump{p_h},\jump{q_h})_{F\setminus \Omega}
    &\quad \text{if }  Q_h = \Pzero,\\
    \beta_3 \sum_{F \in \Fast}
    h_F^{3} (\jump{\partial_{\bfn} p_h}, \jump{\partial_{\bfn} q_h})_{F}
    &\quad \text{if } Q_h = \Pone.
  \end{cases}
  \label{p2:eq:nitsche-fd-ja}
\end{align}
Again, $[v] = v^+ - v^-$ is the jump over each facet $F$, and $\bfn =
\bfn_F$ is a fixed, but arbitrary, unit normal to the facet $F$. Here,
$\beta_2 > 0$ and $\beta_3 > 0$ denote additional penalty
parameters. As before, we are allowed to
rewrite~\eqref{p2:eq:nitsche-fd-ja} as a single form $j_h(p_h, q_h) =
j_{h,0}(p_h, q_h) + j_{h,1}(p_h, q_h)$ with $j_{0,h}$ and $j_{h,1}$
denoting \eqref{p2:eq:nitsche-fd-ja} in the case of $Q_h = \Pzero$ and
$Q_h = \Pone$, respectively.

We are now ready to state the Nitsche based fictitious domain method
for the Stokes problem~\eqref{p2:eq:strongform}.  Let $V_h =
V_h(\meshast)$ and $Q_h = Q_h(\meshast)$ be the finite element
velocity and pressure spaces defined relative to $\meshast$. The
variational problem reads: find $(\bfu_h,p_h) \in V_{h} \times Q_{h}$
such that
\begin{equation}
  \label{p2:eq:stokes-fd}
  A_h(\bfu_h, p_h ; \bfv_h, q_h) + J_h(\bfu_h, p_h; \bfv_h, q_h)
  = L_h(\bfv_h, q_h) \quad \foralls
  (\bfv_h,q_h) \in V_h \times Q_h,
\end{equation}
where $A_h$ and $J_h$ are defined by
\begin{align}
  A_h(\bfu_h,p_h;\bfv_h,q_h) &=
  a_h(\bfu_h,\bfv_h)
  + b_h(\bfu_h,q_h) + b_h(\bfv_h,p_h)
  - c_h(p_h, q_h),
\label{p2:eq:fict-domain-stokes-A}
  \\
  J_h(\bfu_h, p_h; \bfv_h, q_h) &= i_h(\bfu_h, \bfv_h) - j_h(p_h,q_h),
\label{p2:eq:fict-domain-stokes-J}
\end{align}
where the forms $c_h$ and $\Phi_h$ are defined as
in~\eqref{p2:eq:stabilization_term_lhs}
and~\eqref{p2:eq:stabilization_term_rhs} (relative to the cut mesh
$\mesh$). The form $L_h$ is given by
\begin{equation}
  L_h(\bfv_h,q_h) = (\bff,\bfv_h)_{\Omega} + (\bfg, \gamma h^{-1} \bfv_h
  - \nablan \bfv_h + q_h \bfn)_{\pO} - \Phi_h(q_h).
  \label{p2:eq:stokes-fd-rhs}
\end{equation}

\begin{rem}
The ``ghost-penalty'' defined
in~\eqref{p2:eq:normal-derivative-jump-penalty} was introduced by
\citet{BurmanHansbo2011} to formulate a first-order convergent fictitious
domain method for the Poisson problem.  The role of the ghost-penalty
is to augment the bilinear form $a_h$ by accounting for small elements
$|T \cap \Omega| \ll |T|,\; T\in \meshast$ in the vicinity of the
boundary $\pO$.
\end{rem}
\begin{rem}
In the Stokes problem, the stabilization form $c_h$ acting on the
pressure also has to be augmented. Depending on the pressure
discretization, this can be achieved in different ways.  In the case
of $Q_h = \Pzero$, a similar ghost-penalty was presented
by~\citet{BeckerBurmanHansbo2009} to propose a finite element method
for incompressible elasticity problems with discontinuous modulus of
elasticity.  To motivate the ghost-penalty~\eqref{p2:eq:nitsche-fd-ja}
when $Q_h = \Pone$, one may consider the stabilization terms $h_T^2
(\nabla p, \nabla q)_T$ and $h_T^2(\bff,\nabla q)_T$ as a locally
scaled version of a Poisson equation and apply
\eqref{p2:eq:normal-derivative-jump-penalty}. In
Lemma~\ref{lem:l2norm-control-via-jumps}, we will reveal the basic
structure behind the augmentation terms and also present a
generalization to higher-order elements.
\end{rem}

%------------------------------------------------------------------------------
\section{Approximation properties}
\label{p2:sec:approx-est}

Before we proceed with the \apriori{} error analysis of the method
proposed in Section~\ref{p2:ssec:stokes-fd-fem}, we summarize here some
notation and useful inequalities that will be used throughout
Sections~\ref{p2:sec:stab-est} and \ref{p2:sec:apriori-estimate}.  In what
follows, $\mcV_h^{\ast}$ and $\mcV_h$ denote some finite element
spaces consisting of piecewise polynomial functions defined on
$\meshast$ and $\mesh$ respectively, but it should be clear that we
have mainly $\mcV_h = V_h$ or $\mcV_h = Q_h$ in mind. The constants
$C$ involved in the inequalities will only depend on $\Omega$ or
$\Oast$, the regularity of the relevant function spaces, the
shape-regularity of $\meshast$, and possibly the polynomial order of
$\mcV_h$; in particular, the constants $C$ do not depend on $h$.

\subsection{Trace inequalities and inverse estimates}
\label{p2:ssec:trace-inverse}
We recall the following trace inequalities for $v \in H^1(\Oast)$:
\begin{alignat}{1}
  \| v \|_{\partial T} &\leqslant  C (h_{T}^{-1/2} \| v \|_{T}
  + h_{T}^{1/2} \| \nabla v \|_{T} )  \quad \foralls T \in \meshast,
  \label{p2:eq:trace-inequality}
  \\
  \| v \|_{T \cap \pO} &\leqslant  C (h_{T}^{-1/2} \| v \|_{T}
  + h_{T}^{1/2} \| \nabla v \|_{T} )
  \quad \foralls T \in \meshast.
  \label{p2:eq:trace-inequality-for-FD}
\end{alignat}
See~\citet{HansboHansbo2002} for a proof
of~\eqref{p2:eq:trace-inequality-for-FD}.  We will also need the
following well-known inverse estimates for $v_h \in \mcV_h$:
\begin{alignat}{3}
  \| \nabla v_h \|_T &\leqslant C h_T^{-1} \| v_h \|_{T} &\quad
  &\foralls T \in \meshast,
  \label{p2:eq:inverse-estimates-for-triangles}
  \\
  \| h^{1/2} \bfn \cdot \nabla v_h \|_{F} &\leqslant C \| \nabla v_h
  \|_{T} &\quad
  &\foralls T \in \meshast,
  \label{p2:eq:inverse-estimate-for-facets}
\end{alignat}
For proofs, we refer to~\citet{Quarteroni2009}.
Moreover, we will need a version of~\eqref{p2:eq:inverse-estimate-for-facets}
for the boundary parts $\pO \cap T$:
\begin{equation}
  \| h^{1/2} \bfn \cdot \nabla v_h \|_{\pO \cap T} \leqslant C \|
  \nabla v_h \|_{T}
  \quad
  \foralls T \in \meshast,
  \label{p2:eq:inverse-estimate-for-fd-traces}
\end{equation}
which was proved under assumptions similar to G1 -- G2
by~\citet{HansboHansbo2002}.
We note that for $\bfv_h \in V_h$, $q_h \in Q_h$,
we have the two estimates:
\begin{equation}
  \tn \bfv_h \tn \leqslant C \tn \bfv_h \tnast,
  \qquad
  \tn q_h    \tn \leqslant C \tn q_h    \tnast,
  \label{p2:eq:tn-tnast-norm-relation}
\end{equation}
which can easily be deduced
by~\eqref{p2:eq:inverse-estimate-for-fd-traces}, and by
combining~\eqref{p2:eq:trace-inequality-for-FD}
and~\eqref{p2:eq:inverse-estimates-for-triangles}.

\subsection{Interpolation estimates}
\label{p2:ssec:approx-props}

In order to construct an interpolation operator $L^2(\Omega)
\rightarrow \mcV_h$, we recall that there is a linear extension operator
$\mcE:H^s(\Omega) \rightarrow H^s(\Oast)$, $s\geqslant 0$, such that
\begin{equation}\label{p2:eq:stabext}
\| \mcE u \|_{s,\Oast} \leqslant C \| u \|_{s,\Omega}.
\end{equation}
See~\citet{Stein1970} for further details. Let $\piast:L^2(\Oast)
\rightarrow \mcV^{\ast}_h$ be the standard Scott--Zhang interpolation
operator~\citep{ScottZhang1990} and recall the interpolation error
estimates
\begin{alignat}{3}
\| v - \piast v \|_{r,T} &\leqslant C h^{s-r}| v |_{s,\omega(T)},
&\quad 0\leqslant r \leqslant s \leqslant 2 \quad &\foralls T\in \meshast,
  \label{p2:eq:interpest0}
  \\
\| v - \piast v \|_{r,F} &\leqslant C h^{s-r-1/2}| v |_{s,\omega(T)},
&\quad 0\leqslant r \leqslant s \leqslant 2 \quad &\foralls F\in \partial_i \meshast,
  \label{p2:eq:interpest1}
\end{alignat}
where $\omega(T)$ is the patch of neighbors of element $T$; that is,
the domain consisting of all elements sharing a vertex with $T$. Next,
we define $\pi_{h}:L^2(\Omega) \rightarrow \mcV^{\ast}_{h}$ as follows:
\begin{equation}
\pi_{h} v = \piast \mcE v .
\label{p2:eq:interpolant-fd-definition}
\end{equation}
Note that $\pi_{h} v$ is now defined on $\Oast$, and in particular on
$\Omega \subset \Oast$.

The stability estimate~\eqref{p2:eq:stabext} together with the interpolation
error estimates~\eqref{p2:eq:interpest0} and~\eqref{p2:eq:interpest1} for
the Scott--Zhang interpolation operator imply the following
interpolation estimates:
\begin{alignat}{3}
\| v - \pi_h v \|_{r,T} &\leqslant C h^{s-r}| v |_{s,\omega(T)},
&\quad 0\leqslant r \leqslant s \leqslant 2 \quad &\foralls T\in \mesh,
  \label{p2:eq:interpest0-cutmes}
  \\
  \label{p2:eq:interpest1-cutmes}
\| v - \pi_h v \|_{r,F} &\leqslant C h^{s-r-1/2}| v |_{s,\omega(T)},
&\quad 0\leqslant r \leqslant s \leqslant 2 \quad &\foralls F\in \partial_i \mesh.
\end{alignat}

We now return to our specific finite elements spaces $V_h$ and $Q_h$.
For the energy norm, we have the following interpolation error
estimates:
\begin{lem}
  \label{p2:lem:interpest-vp}
  For the interpolation operator $\pi_h$ defined
  by~\eqref{p2:eq:interpolant-fd-definition}, there is a constant $C >
  0$ such that for all $\bfv \in [H^2(\Omega)]^d$ and all $q \in
  H^1(\Omega)$:
  \begin{align}
    \tn \bfv - \pi_h \bfv \tn &\leqslant C h  | \bfv |_{2, \Omega},
    \label{p2:eq:interpest-v}
    \\
    \tn (\bfv - \pi_h \bfv, q - \pi_h q) \tn &\leqslant C h ( | \bfv
    |_{2,\Omega} + | q |_{1,\Omega} ).
    \label{p2:eq:interpest-vp}
  \end{align}
\end{lem}
\begin{proof}
  We only sketch the proof. First use the trace
  inequality~\eqref{p2:eq:trace-inequality-for-FD} to estimate the
  boundary contributions in terms of element contributions. Then apply
  the interpolation error estimate~\eqref{p2:eq:interpest0}, and
  finally the stability estimate~\eqref{p2:eq:stabext}.
\end{proof}

In addition to the interpolation estimates, we will need the following
continuity property of the extended interpolation operator with
respect to different norms:
\begin{lem}
  \label{lem:interpolant-continuity}
  Assume $\bfv \in [H^1_0(\Omega)]^d$ and let $\pi_h: [H^1(\Omega)]^d
  \rightarrow V_h(\meshast)$ be the interpolation operator defined
  in~\eqref{p2:eq:interpolant-fd-definition}.  Then there is a constant
  $C > 0$ such that
  \begin{align}
    \tn \pi_h \bfv \tn_{\ast} \leqslant C\| \bfv \|_{1,\Omega}.
    \label{p2:eq:interpolant-continuity}
  \end{align}
\end{lem}
\begin{proof}
  By definition we have $\tn \pi_h \bfv \tnast^2 = \| \nabla \pi_h
  \bfv \|^2_{\Oast} + \|h^{-1/2} \pi_h \bfv \|^2_{\pO}$.  The bound
  for the first term on the right-hand side follows immediately by the
  boundedness of $\piast$ and the continuity of the extension operator
  $\mcE$.  To estimate the second term, we use the fact that $\mcE
  \bfv|_{\pO} = 0$ for $\bfv \in [H^1_0(\Omega)]^d$, the trace
  inequality~\eqref{p2:eq:trace-inequality-for-FD}, the interpolation
  estimate~\eqref{p2:eq:interpest0} and continuity of $\mcE$
  again:
  \begin{align*}
    \| h^{-1/2} \pi_h \bfv \|_{\Gamma}^2
    &= \sum_{T \in \meshast_{\pO}} h_T^{-1} \| \pi_h \bfv \|_{\pO \cap T}^2
    = \sum_{T \in \meshast_{\pO}} h_T^{-1} \| \pi_h \bfv - \mcE \bfv \|_{\pO \cap T}^2 \\
    &\leqslant
    \sum_{T \in \meshast_{\pO}}
    h_T^{-1} \left ( h_T^{-1} \| \pi_h \bfv - \mcE \bfv \|_{ T}^2
    + h_T \| \nabla \left ( \pi_h \bfv - \mcE \bfv \right ) \|_{ T}^2 \right ) \\
    &\leqslant C \| \mcE \bfv \|_{1,\Oast}^2
    \leqslant C \| \bfv \|_{1, \Omega}^2 .
  \end{align*}
\end{proof}

%------------------------------------------------------------------------------
\section{Stability estimates}
\label{p2:sec:stab-est}

In this section, we demonstrate that the bilinear form defining the
stabilized Nitsche fictitious domain variational
formulation~\eqref{p2:eq:stokes-fd} indeed satisfies the inf-sup
stability condition in the Babu\v{s}ka--Brezzi sense.

\subsection{The role of the boundary zone jump-penalties}
\label{p2:ssec:role-jump-penalities}

\begin{figure}
  \begin{center}
    \includegraphics[width=0.65\textwidth]{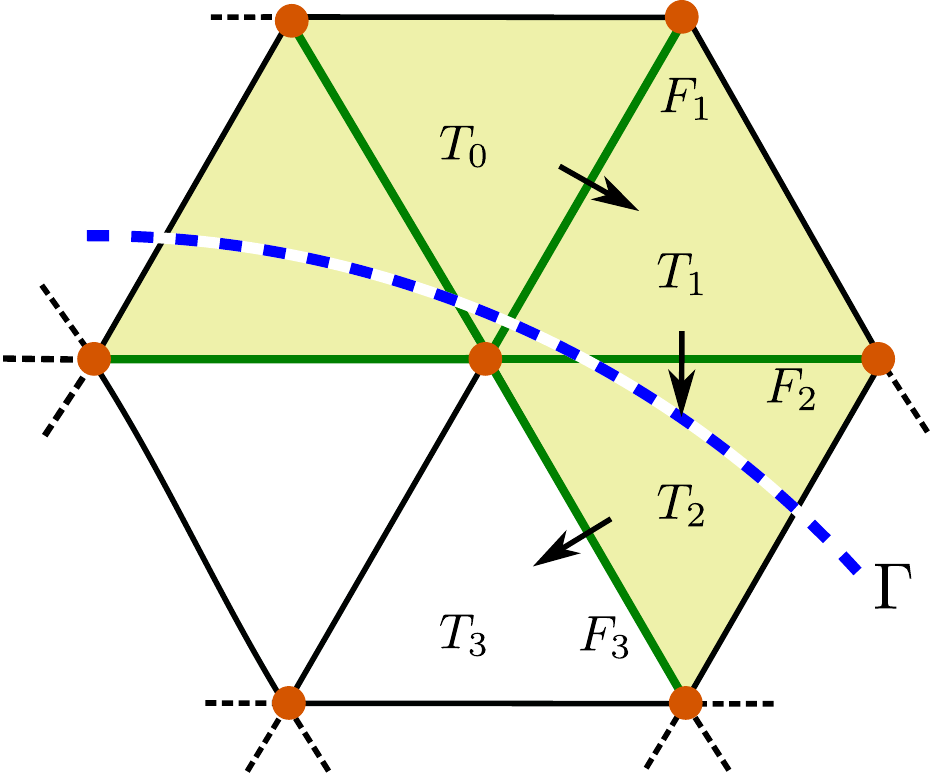}
    \caption{Controlling the $L^2$-norm $\| v \|_{T_0}$
    of a finite element function $v$
    on a barely intersected, ``fictitious'' element $T_0$
    by $\| v \|_{T_3}$ and boundary zone jump-penalties.
    Starting from $T_0$, each term $\|v\|_{T_i}^2$
    can be estimated by the neighboring term
    $\| v \|_{T_{i+1}^2}$ when a sum of
    jump-terms of the form $h_{F_{i+1}}^{2j+i} \|\nablan v \|_{F_{i+1}}^2$
    is added.}
  \label{fig:boundary-zone-zoom}
  \end{center}
\end{figure}

As a first step, we show how the
jump-penalties~\eqref{p2:eq:normal-derivative-jump-penalty}
and~\eqref{p2:eq:nitsche-fd-ja} in the boundary zone $\Fast$ contribute
to control the norms of $\bfv_h$ and $p_h$ on the \emph{entire}
fictitious domain $\meshast$.  We start with the following lemma.

\begin{lem}
  \label{lem:l2norm-control-via-jumps}
  Let $\mesh = \{T\}$ be a tessellation consisting of shape-regular
  elements $T$ and let $T_1,\, T_2 \in \mesh$ be two elements sharing
  a common face $F$. Furthermore, let $v$ be a piecewise polynomial
  function defined relative to the macro-element $\overline{T}= T_1
  \cup T_2$.  Let $v_i$ be the restriction of $v$ to $T_i$ for $i = 1,
  2$. Then there is a constant $C > 0$, depending only on the
  shape-regularity of $\mesh$ and the polynomial order $p =
  \max(\ord(v_1), \ord(v_2))$ of $v$, such that
  \begin{equation}
    \| v \|_{T_1}^2 \leqslant C \left ( \|v\|_{T_2}^2
    + \sum_{j \leqslant p} h^{2 j + 1}
    (\jump{\nablan^{j}
    v},\jump{\nablan^{j} v})_F \right ),
    \label{p2:eq:l2norm-control-via-jumps}
  \end{equation}
  where $\nablan^j v = \sum_{| \alpha | = j}D^{\alpha} v(\bfx)
  \bfn^{\alpha}$ for multi-index $\alpha = (\alpha_1, \ldots,
  \alpha_d)$, $|\alpha| = \sum_{i} \alpha_i$ and $\bfn^{\alpha} =
  n_1^{\alpha_1} n_2^{\alpha_2} \cdots n_d^{\alpha_d}$.
\end{lem}
\begin{proof}
  For a given point $\bfx \in T_1$, we write $\bfx_F = \bfx_F(\bfx)$
  for the normal projection of $\bfx$ onto the plane defined by the
  face $F$. Note that the area $|T_{1, F}|$ of all projected points in
  $T_1$ is bounded by $|F|$ up to a constant by the shape-regularity
  assumption. For $i = 1, 2$, and since $\ord(v_i) \leqslant p$, we may
  express (the extensions to $\overline{T}$ of) $v_i$ in terms of its
  Taylor-expansion around $\bfx_F$:
  \begin{equation*}
    v_i(\bfx) = \sum_{|\alpha| \leqslant p}
    \dfrac{D^{\alpha}v_i(\bfx_F)}{\alpha !} (|\bfx-\bfx_F| \bfn )^{\alpha},
  \end{equation*}
  where $\bfn$ is the unit normal vector of $F$ pointing towards
  $T_1$. Subtracting the two Taylor expansions, we find that
  \begin{equation*}
    v_1(\bfx) = v_2(\bfx)
    + \sum_{|\alpha| \leqslant p}
    \dfrac{\jump{D^{\alpha}v(\bfx_F)}}{\alpha !} (|\bfx-\bfx_F| \bfn )^{\alpha}.
  \end{equation*}
  Next, integrating over $T_1$ with respect to $\bfx$, taking squares
  and applying the Cauchy--Schwarz inequality yield
  \begin{equation*}
    \| v_1 \|_{T_1}^2 \leqslant C \left (
    \|v_2\|_{T_1}^2 + \sum_{|\alpha| \leqslant p}
    \int_{T_1}  \left ( \jump{D^{\alpha} v(\bfx_F(\bfx))} \bfn^{\alpha} \right )^{2}
    h^{2 |\alpha|} \dx \right ),
  \end{equation*}
  with $h$ the maximal element diameter. From the assumption of shape
  regularity, a change of variables, and the definition of
  $\nablan^j$, it follows that
  \begin{equation*}
    \| v_1 \|_{T_1}^2 \leqslant C \left (\|v_2\|_{T_1}^2
    + \sum_{j \leqslant p}
    \int_{F}{\jump{\partial_{\bfn}^jv (y)}^2}
    h^{2j + 1} \, \mathrm{d}y \right ).
  \end{equation*}
  Finally, as the two norms $\|v_2\|_{T_1}$ and $\|v_2\|_{T_2}$ are
  equivalent, again by shape regularity, we obtain the desired
  inequality~\eqref{p2:eq:l2norm-control-via-jumps}.
\end{proof}

\begin{rem}
  The previous lemma is a key observation for proving stability and
  \apriori{} error estimates for the fictitious domain
  formulation~\eqref{p2:eq:stokes-fd} as it lays the foundation for how
  to control certain norms on the fictitious domain $\Oast$ in terms
  of norms computed only on $\Omega$ and appropriate jump-penalties in
  the intersection zone $\Fast$.
\end{rem}
We are now in a position to state the following proposition:
\begin{prop} \label{prop:control-norm-with-ghost-penalties}
  Let $\Omega$, $\Oast$ and $\Fast$ be defined as in
  Section~\ref{p2:sec:notation-assumptions}. There is a constant $C >
  0$ such that the following estimates hold. \\ For all $\bfv_h \in
  V_h(\meshast)$:
  \begin{equation}
  \| \nabla \bfv_h \|_{\Oast}^2
  \leqslant C \bigl( \| \nabla \bfv_h
  \|_{\Omega}^2
  + \sum_{F \in \Fast} h_F
  (\jump{\nablan \bfv_h}, \jump{\nablan \bfv_h})_F
  \bigr)
  \leqslant C \| \nabla \bfv_h \|_{\Oast}^2
  \label{p2:eq:u-fd-norm-est}
\end{equation}
  and for all $q_h \in \Pzero(\meshast)$:
\begin{equation}
  \| q_h \|_{\Oast}^2
  \leqslant C
  \bigl (
  \| q_h \|_{\Omega}^2
  + \sum_{F \in  \Fast} h_F
  (\jump{\bfn \cdot q_h}, \jump{\bfn \cdot  q_h})_F
  \bigr )
  \leqslant C \| q_h \|_{\Oast}^2
  \label{p2:eq:p0-fd-norm-est}
\end{equation}
while, for all $q_h \in \Pone(\meshast)$:
\begin{equation}
  \| q_h \|_{\Oast}^2 \leqslant C
  \bigl (
  \| q_h \|_{\Omega}^2
  + \sum_{F \in  \Fast} h_F^3
  (\jump{\nablan q_h}, \jump{\nablan q_h})_F
  \bigr )
  \leqslant C \| q_h \|_{\Oast}^2 .
  \label{p2:eq:p1-fd-norm-est}
\end{equation}
\end{prop}
\begin{proof}
  We start with the first inequality
  of~\eqref{p2:eq:u-fd-norm-est}. Decompose the norm over $\Oast$ into
  sums over non-cut and cut elements. Let $T_0 \in \meshast_{\pO}$ be
  a cut element.  By the geometric condition G3
  (cf.~Section~\ref{p2:sec:notation-assumptions}), there exists a
  $T_N \subset \Omega$ and at most $N-1$ elements
  $T_i \in \meshast_{\pO}$ and facets $T_{i-1} \cap T_{i} =
  F_i \in \Fast$ that have to be crossed in order to traverse from
  $T_0$ to $T_N$.  By the shape-regularity of the mesh, each facet
  $F \in \Fast$ will only be involved in a finite number of such
  crossings. Applying Lemma~\ref{lem:l2norm-control-via-jumps}, with
  each component of $\nabla \bfv_h$ as $v$, iteratively to each
  neighboring pair $\{T_{i-1},T_i\}$ yields the desired estimate.

  The first inequalities of~\eqref{p2:eq:p0-fd-norm-est}
  and~\eqref{p2:eq:p1-fd-norm-est} follow by the analogous argument:
  apply Lemma~\ref{lem:l2norm-control-via-jumps} to $q_h$ and recall
  that $\jump{q_h}_F = 0$ for $q_h \in \Pone$.

  The second inequalities
  of~\eqref{p2:eq:u-fd-norm-est}--\eqref{p2:eq:p1-fd-norm-est} rely on the
  shape regularity, allowing us to bound $h_F$ by $h$, and the trace
  and inverse estimates of Section~\ref{p2:ssec:trace-inverse} applied to
  each facet of the boundary zone sums. The upper bounds immediately
  follow.
\end{proof}

\begin{rem}
  \citet{BurmanHansbo2011} presented the analogous result
  to~\eqref{p2:eq:u-fd-norm-est} for the Poisson problem with continuous
  piecewise linear finite elements.  The formulation given here,
  together with Lemma~\ref{lem:l2norm-control-via-jumps}, reveals the
  basic structure of jump-penalty-based stabilization terms for
  fictitious domain formulations and can be applied to various types
  of norms and elements, including higher-order elements.
\end{rem}

\subsection{Stability estimates and the inf-sup condition}

The main result of this section, Theorem~\ref{thm:stability}, is the
inf-sup stability of the bilinear form $A_h + J_h$, occurring in the
stabilized Nitsche fictitious domain variational
formulation~\eqref{p2:eq:stokes-fd}, with respect to the $\tn \cdot
\tn_{*}$ norm~\eqref{p2:eq:triple-up-norm}.

We begin by establishing the properties of the separate contributions
to the bilinear form.  First, the form $a_h$
cf.~\eqref{p2:eq:nitsche-fd-form-a} augmented by $i_h$
cf.~\eqref{p2:eq:normal-derivative-jump-penalty} is continuous and
coercive with respect to the norms $\tn \cdot \tn$ and $\tn \cdot
\tn_{\ast}$~\citep{BurmanHansbo2011}. More precisely,
there are constants $c > 0$ and $C > 0$ such that
\begin{align}
  a_h(\bfv,\bfw) &\leqslant C \tn \bfv \tn \; \tn \bfw \tn \quad \foralls
  \bfv,\bfw \in [H^1(\mesh)]^d,
  \label{p2:eq:a_h-stable-1}
  \\
  a_h(\bfv_h,\bfw_h) + i_h(\bfv_h,\bfw_h) &\leqslant C \tn \bfv_h \tn_{\ast} \;
  \tn \bfw_h \tn_{\ast} \quad \foralls
  \bfv_h,\bfw_h \in V_h(\meshast),
  \label{p2:eq:a_h-stable-2}
  \\
  c \tn \bfw_h \tn_{\ast} ^2 &\leqslant a_h(\bfw_h, \bfw_h) + i_h(\bfw_h,
  \bfw_h) \quad \foralls \bfw_h \in V_h(\meshast).
  \label{p2:eq:a_h-stable-3}
\end{align}

Next, we show that $b_h$ is continuous with respect to the relevant
norms.
\begin{prop}
  \label{prop:b_h-stabprop}
  Let $b_h$ be defined by~\eqref{p2:eq:nitsche-fd-form-b}. There is a
  constant $C > 0$ such that
  \begin{alignat}{2}
    b_h(\bfv, q) &\leqslant C \tn \bfv \tn \, \tn q \tn &\quad &\foralls
    (\bfv,q) \in [H^1_0(\Omega)]^d \times L^2(\Omega),
    \label{p2:eq:b-cont-wrt-norm-1}
    \\
    b_h(\bfv_h, q_h) &\leqslant C \tn \bfv_h \tn_{\ast} \tn q_h \tn_{\ast}
    &\quad &\foralls (\bfv_h,q_h) \in V_h(\meshast) \times Q_h(\meshast).
    \label{p2:eq:b-cont-wrt-norm-2}
  \end{alignat}
\end{prop}
\begin{proof}
  The bound~\eqref{p2:eq:b-cont-wrt-norm-2} follows from the
  definitions of $b_h$ and the $\tn~\cdot~\tn_{\ast}$ norm, and a
  subsequent use of~\eqref{p2:eq:trace-inequality-for-FD}
  and~\eqref{p2:eq:inverse-estimates-for-triangles}.
\end{proof}

The next lemma gives a fictitious domain adapted version of a ``bad
inequality'' often used in Verf\"urth's trick \citep{Verfurth1994a}
and in proofs for some classical, stabilized
schemes~\citep{Franca1993}.
\begin{lem}
  \label{p2:lem:bad-inequality}
  There are positive constants $C_1, C_2$ such that for each $q_h \in
  Q_h(\mesh)$ there exists a $\bfv_h \in V_h(\meshast)$ satisfying
  \begin{equation}
    \frac{b_h(\bfv_h, q_h)}{\tn \bfv_h \tnast}
    \geqslant
    C_1 \| q_h \|_{\Omega} -
    C_2 \big ( (\sum_{T \in \mesh} h_T^2 \| \nabla q_h \|_{T}^2 )^{1/2}
             + (\sum_{F \in \partial_i \mesh} h_F \| \jump{q_h} \|_{F}^2 )^{1/2}
       \big ) .
    \label{p2:eq:bad-inequality}
  \end{equation}
\end{lem}
\begin{proof}
  Let $q_h \in Q_h(\mesh)$ be given. There exists a $\bfv \in
  [H^1(\Omega)]^d$ and a constant $\widetilde{C}_1 > 0$ such that
  $\Div \bfv = q_h$ and $ \widetilde{C}_1 \| \bfv \|_{1,\Omega} \leqslant \|
  \Div \bfv \|_{\Omega}$~\citep{GiraultRiviereWheeler2005}. Map $\bfv
  \mapsto \pi_h \bfv \in V_h(\meshast)$ by the extended interpolation
  operator cf.~\eqref{p2:eq:interpolant-fd-definition}, and denote
  $\bfe_h = \pi_h \bfv - \bfv$.  It follows, using the definition of
  $b_h$, that
  \begin{equation}
    \label{p2:eq:lem:step1}
    b_h(\pi_h \bfv,q_h)
    = b_h(\bfe_h, q_h) + b_h(\bfv, q_h)
    \geqslant b_h(\bfe_h, q_h)
    + \widetilde{C}_1 \|\bfv\|_{1,\Omega} \| q_h \|_{\Omega} .
  \end{equation}
  Moreover, integrating by parts on each element $T \in \mesh$ yields
  \begin{equation}
    \label{p2:eq:lem:step2}
    \begin{split}
    b_h(\bfe_h, q_h)
    &= - (\Div \bfe_h, q_h)_{\Omega} + ( \bfn \cdot \bfe_h, q_h)_{\pO} \\
    &= \sum_{T \in \mesh} (\bfe_h, \nabla q_h)_{T}
    + \sum_{F \in \partial_i \mesh} (\bfn \cdot \bfe_h,\jump{q_h})_F,
    \end{split}
  \end{equation}
  while the Cauchy-Schwarz inequalities give
  \begin{align}
    \label{p2:eq:lem:step3:a}
    \sum_{T \in \mesh} (\bfe_h, \nabla q_h)_{T}
    &\geqslant
    - (\sum_{T \in \mesh} h_T^{-2} \| \bfe_h \|^2_{T})^{1/2}
    (\sum_{T \in \mesh} h_T^{2} \| \nabla q_h \|_{T}^2)^{1/2}, \\
    \label{p2:eq:lem:step3:b}
    \sum_{F \in \partial_i \mesh} (\bfn \cdot \bfe_h,\jump{q_h})_F
    &\geqslant - (\sum_{F \in \partial_i \mesh}h_F^{-1} \| \bfe_h \|^2_{F})^{1/2}
    (\sum_{F \in \partial_i \mesh} h_F \| \jump{q_h} \|_{F}^2)^{1/2}.
  \end{align}
  Since $h_T^{-1} \|\bfe_h\|_{T} \leqslant C \|\bfv\|_{1, \omega(T)}$
  by~\eqref{p2:eq:interpest0-cutmes} and $h_F^{-1/2}
  \|\bfe_h\|_{F} \leqslant C \|\bfv\|_{1, \omega(T)}$
  by~\eqref{p2:eq:interpest1-cutmes}, we obtain by
  combining~\eqref{p2:eq:lem:step2} with~\eqref{p2:eq:lem:step3:a}
  and~\eqref{p2:eq:lem:step3:b}:
  \begin{equation}
    \label{p2:eq:lem:step4}
    \begin{split}
      b_h(\bfe_h, q_h)
      \geqslant - \widetilde{C}_2 \|\bfv \|_{1, \Omega} \left (
      (\sum_{T \in \mesh} h_T^2 \| \nabla q_h \|_{T}^2 )^{1/2}
      + (\sum_{F \in \partial_i \mesh} h_F \| \jump{q_h} \|_{F}^2)^{1/2} \right )
    \end{split}
  \end{equation}
  Finally, combining~\eqref{p2:eq:lem:step1}
  with~\eqref{p2:eq:lem:step4}, and recalling that $\tn \pi_h \bfv
  \tnast \leqslant C \| \bfv \|_{1, \Omega}$ by
  Lemma~\ref{lem:interpolant-continuity},
  yields~\eqref{p2:eq:bad-inequality} with $\bfv_h = \pi_h \bfv$.
\end{proof}

Using the stability estimates for $a_h$ and $b_h$, we may now prove
the following inf-sup stability estimate for for $A_h + J_h$:
\begin{thm}
  There is a constant $c_A > 0$ such that for all $(\bfu_h,p_h) \in
  V_h \times Q_h = V_h(\meshast) \times Q_h(\meshast)$:
  \begin{equation}
    \sup_{(\bfv_h, q_h) \in V_h \times Q_h }
    \frac{A_h(\bfu_h, p_h; \bfv_h, q_h ) + J_h(\bfu_h, p_h; \bfv_h, q_h)}{\tn (\bfv_h,q_h) \tn_{\ast}}
    \geqslant  c_A \tn (\bfu_h,p_h) \tnast.
    \label{p2:eq:stability}
  \end{equation}
  \label{thm:stability}
\end{thm}
\begin{proof}
  The proof of Theorem~\ref{thm:stability} follows the proof
  by~\citet{Franca1993}, using the appropriate norms and
  Proposition~\ref{prop:control-norm-with-ghost-penalties} in
  combination with Lemma~\ref{p2:lem:bad-inequality}. Let $(\bfu_h,
  p_h)$ be given.

  First, choose $(\bfv_h,q_h)$ to be $(-\bfw_h,0)$ where $-\bfw_h$
  satisfies~\eqref{p2:eq:bad-inequality} for the given $p_h$. In
  addition, scale $\bfw_h$ such that $\tn \bfw_h \tnast = \| p_h
  \|_{\Omega}$.  For the sake of readability, we write
  \begin{equation}
    k_{\mesh}(p_h) \equiv (\sum_{T \in \mesh} h_T^2 \| \nabla p_h \|_{T}^2
    )^{1/2} + (\sum_{F \in \partial_i \mesh} h_F \| \jump{p_h}
    \|_{F}^2)^{1/2} .
  \end{equation}
  With this choice of test functions,
  applying~\eqref{p2:eq:a_h-stable-2}
  and~\eqref{p2:eq:bad-inequality}, and Cauchy's inequality with
  $\epsilon$ give
  \begin{align*}
    (A_h + J_h)&(\bfu_h, p_h; -\bfw_h,0)
    = - a_h(\bfu_h,\bfw_h) - i_h(\bfu_h, \bfw_h) + b_h(\bfw_h, -p_h)
    \\
    &\geqslant
    - C \tn \bfu_h \tnast \, \| p_h \|_{\Omega}
    + C_1 \| p_h \|_{\Omega}^2
    - C_2 k_{\mesh}(p_h) \, \| p_h \|_{\Omega} \\
    &\geqslant - \frac{C}{4 \epsilon} \tn \bfu_h \tnast^2
    + \left ( C_1 - \epsilon (C + C_2) \right ) ||p_h||_{\Omega}^2
    - \frac{C_2}{4 \epsilon} k_{\mesh}(p_h)^2 .
  \end{align*}
  Note that by definition $k_{\mesh}(p_h)^2 \leqslant K c_h(p_h, p_h)$ for
  some positive constant $K$ depending on $\beta_0$ and $\beta_1$. In
  combination with choosing $\epsilon$ such that $(C_1 - \epsilon (C +
  C_2)) > 0$, this gives
  \begin{equation*}
    (A_h + J_h)(\bfu_h, p_h; -\bfw_h,0)
    \geqslant
    - \tilde{C} \tn \bfu_h \tnast^2
    + \tilde{C}_1 \| p_h \|_{\Omega}^2
    - \tilde{C}_2 c_h(p_h, p_h) .
  \end{equation*}

  Second, we take test functions $(\bfv_h,q_h) = (\bfu_h, -p_h)$
  which, using~\eqref{p2:eq:a_h-stable-3}, gives
  \begin{align*}
    (A_h + J_h)(\bfu_h,p_h;\bfu_h,-p_h)
    &= a_h(\bfu_h,\bfu_h) + i_h(\bfu_h,\bfu_h) + c_h(p_h, p_h) + j_h(p_h,p_h)
     \\
    &\geqslant c \tn \bfu_h \tnast^2 + c_h(p_h, p_h) + j_h(p_h,p_h) .
  \end{align*}

  In total, for any $\delta > 0$, we have
  \begin{align*}
    (A_h + J_h)&(\bfu_h,p_h;\bfu_h,-p_h) + \delta(A_h + J_h)(\bfu_h,p_h;-\bfw_h,0)
    \\
    & \geqslant (c - \delta \tilde C) \tn \bfu_h \tnast^2
    + (1 - \delta \tilde C_2) c_h(p_h, p_h)
    + \delta \tilde{C}_1 \| p_h \|_{\Omega}^2 + j_h(p_h, p_h) .
  \end{align*}
  Moreover, by~\eqref{p2:eq:p0-fd-norm-est} and~\eqref{p2:eq:p1-fd-norm-est},
  there exists a positive constant $D$ such that
  \begin{equation}
    \|p_h \|_{\Omega}^2 + j_h(p_h, p_h)
    \geqslant D \|p_h \|_{\Oast}^2 \equiv D \tn p_h \tnast^2 .
  \end{equation}
  Finally, we conclude that by a suitable choice of $\delta > 0$,
  there is a positive constant $c_A$ such that $(\bfv_h, q_h) =
  (\bfu_h-\delta \bfw_h,-p_h)$ satisfies
  \begin{equation*}
  (A_h + J_h)(\bfu_h,p_h;\bfv, q_h) \geqslant c_A (\tn \bfu_h \tnast^2 + \tn
  p_h \tn_{\ast}^2),
  \end{equation*}
  which proves the desired estimate.
\end{proof}

%------------------------------------------------------------------------------
\section{\emph{A priori} error estimate}
\label{p2:sec:apriori-estimate}

Before we formulate the main \apriori{} error estimate, we state two
lemmas about how the stabilization form $J_h$ affects the Galerkin
orthogonality and the consistency of the total form $A_h + J_h$. Let
$V_h = V_h(\meshast)$ and $Q_h = Q_h(\meshast)$ throughout this
section.

\begin{lem}
  (Weak Galerkin orthogonality). Let $(\bfu, p) \in [H^2(\Omega)]^d
  \times H^1(\Omega)$ be the solution of the Stokes
  problem~\eqref{p2:eq:strongform} and let $(\bfu_h,p_h)$ be the
  discrete solution of the corresponding stabilized Nitsche fictitious
  domain formulation~\eqref{p2:eq:stokes-fd}. Then,
  \begin{equation}
    A_h(\bfu - \bfu_h, p - p_h; \bfv_h, q_h)
      - J_h(\bfu_h, p_h;\bfv_h, q_h) = 0
      \quad \foralls (\bfv_h, q_h) \in V_h \times Q_h.
      \label{p2:eq:weak-orthogonality}
  \end{equation}
\end{lem}
\begin{proof}
  The identify follows immediately from the fact that the solution
  $(\bfu,p)$ satisfies $A_h(\bfu,p;\bfv_h,q_h) = L_h(\bfv_h, q_h)$, as
  defined by~\eqref{p2:eq:fict-domain-stokes-A}
  and~\eqref{p2:eq:stokes-fd-rhs}, for all $(\bfv_h,q_h) \in V_h
  \times Q_h$.
\end{proof}

The ghost penalty part $j_{1,h}$ in $J_h$ involves
the evaluation of $\nablan q_h$ on facets and therefore
the variational formulation~\eqref{p2:eq:stokes-fd} is per~se
not consistent with \eqref{p2:eq:strongform} since we only
assume that $q \in H^1(\Omega)$.
The next lemma shows that
this consistency error will not affect the convergence order.

\begin{lem}
  (Weak consistency) Assume that $\bfu \in [H^2(\Omega)]^d$ and $p \in
  H^1(\Omega)$ and let $\pi_h$ be the interpolation
  operator defined by~\eqref{p2:eq:interpolant-fd-definition}. Then
  for all $(\bfv_h,q_h) \in V_h \times Q_h$ it holds that
  \begin{equation}
    J_h(\pi_h \bfu, \pi_h p; \bfv_h, q_h) \leqslant C h(|\bfu|_{2,\Omega} +
    |p|_{1,\Omega}) \tn (\bfv_h, q_h ) \tn_{\ast}.
    \label{p2:eq:weak-consistency}
  \end{equation}
\end{lem}
\begin{proof}
  By definition~\eqref{p2:eq:fict-domain-stokes-J}:
  \begin{equation*}
    J_h(\pi_h \bfu, \pi_h p; \bfv_h, q_h)
    = i_h(\pi_h \bfu, \bfv_h) - j_{h,0}(\pi_h p, q_h) - j_{h,1}(\pi_h p, q_h).
  \end{equation*}
  By the continuity assumption on $\bfu$, $i_h(\mcE \bfu, \bfv_h) =
  0$. So, by the definition of
  $\pi_h$~\eqref{p2:eq:interpolant-fd-definition}, the inverse
  inequality~\eqref{p2:eq:inverse-estimate-for-facets} and the
  interpolation estimate~\eqref{p2:eq:interpest0}, and the continuity
  of $\mcE$, we obtain
  \begin{equation*}
    i_h(\pi_h \bfu, \bfv_h)
    = i_h(\piast \mcE \bfu - \mcE \bfu, \bfv_h)
    \leqslant C h | \mcE \bfu |_{2, \Oast} \| \nabla \bfv_h \|_{\Oast}
    \leqslant C h | \bfu |_{2, \Omega} \tn \bfv_h \tnast .
  \end{equation*}
  Similarly, by the continuity assumption on $p$; the trace
  inequality~\eqref{p2:eq:trace-inequality}, the inverse
  estimate~\eqref{p2:eq:inverse-estimates-for-triangles} and the
  interpolation estimate~\eqref{p2:eq:interpest1}; and the continuity
  of $\mcE$:
  \begin{equation*}
    j_{h,0}(\pi_h p, q_h)
    = j_{h, 0}(\piast \mcE p - \mcE p, q_h)
    \leqslant C h |\mcE p|_{1, \Oast} \| q_h \|_{\Oast}
    \leqslant C h |p|_{1, \Omega} \tn q_h \tnast .
  \end{equation*}

  Finally, to estimate $j_{h,1}$, we
  use~\eqref{p2:eq:inverse-estimate-for-facets}
  and~\eqref{p2:eq:inverse-estimates-for-triangles}; the boundedness
  of the Scott--Zhang interpolant, and the continuity of $\mcE$ to
  obtain
  \begin{equation*}
    j_{h,1}(\pi_h p, q_h)
    \leqslant C h \| \nabla \pi_h p \|_{\Oast} \| q \|_{\Oast}
    \leqslant C h | p |_{1, \Omega} \tn q_h \tnast .
  \end{equation*}
  Combining the three estimates yields the
  result~\eqref{p2:eq:weak-consistency}.
\end{proof}

\begin{thm}
  \label{thm:a-priori-estimate}
  (\emph{A priori} error estimate) Let $(\bfu, p) \in [H^2(\Omega)]^d
  \times H^1(\Omega)$ be the solution of the Stokes problem
  \eqref{p2:eq:strongform} and let $(\bfu_h,p_h)$ be the discrete
  solution of the corresponding stabilized Nitsche fictitious domain
  formulation \eqref{p2:eq:stokes-fd}. Then, there is a constant $C >
  0$ such that
  \begin{equation}
    \tn (\bfu - \bfu_h , p - p_h)\tn
    \leqslant C h \left( | \bfu |_{2,\Omega} + | p |_{1,\Omega} \right) .
    \label{p2:eq:a-priori-estimate}
  \end{equation}
\end{thm}
\begin{proof}
  Clearly, by the triangle inequality
  and~\eqref{p2:eq:tn-tnast-norm-relation}:
  \begin{equation*}
    \tn (\bfu  - \bfu_h, p - p_h ) \tn
    \leqslant \tn (\bfu  - \pi_h \bfu, p - \pi_h p ) \tn
    + C \tn (\pi_h \bfu  - \bfu_h, \pi_h p - p_h ) \tnast .
  \end{equation*}
  Lemma~\ref{p2:lem:interpest-vp} provides the desired bound for the
  first term on the right-hand side above. It is therefore enough to
  show that the discrete error $(\pi_h \bfu - \bfu_h,\pi_h p - p_h)$
  satisfies the error bound in~\eqref{p2:eq:a-priori-estimate}.

  By Theorem~\ref{thm:stability}, there exists a $(\bfv_h, p_h)$ such
  that $\tn (\bfv_h, p_h) \tnast = 1$ and
  \begin{multline*}
    c_A \tn (\pi_h \bfu  - \bfu_h, \pi_h p - p_h )\tnast \\
    \leqslant A_h(\bfu_h - \pi_h \bfu,p_h - \pi_h p;\bfv_h,q_h)
    + J_h(\bfu_h -\pi_h \bfu,  p_h - \pi_h p;\bfv_h,q_h)
    \\
    = A_h(\bfu - \pi_h \bfu;p - \pi_h p;\bfv_h,q_h)
     - J_h(\pi_h \bfu, \pi_h p;\bfv_h,q_h),
  \end{multline*}
  where the last equality follows by the weak Galerkin orthogonality
  \eqref{p2:eq:weak-orthogonality}. Recalling the definition of $A_h$,
  we may write
  \begin{align*}
    A_h(\bfu - \pi_h \bfu,p - \pi_h p;\bfv_h,q_h)
    &= a_h(\bfu - \pi_h \bfu, \bfv_h)
    \\
    &\quad+ b_h(\bfu - \pi_h \bfu, q_h)
    + b_h(\bfv_h, p - \pi_h p)
    \\
    &\quad + c_h(\bfu - \pi_h \bfu;  p - \pi_h p, q_h).
  \end{align*}
  We use the stability estimate~\eqref{p2:eq:a_h-stable-1} for $a_h$
  and \eqref{p2:eq:b-cont-wrt-norm-1} for $b_h$;
  and~\eqref{p2:eq:tn-tnast-norm-relation}
  and~\eqref{p2:eq:interpest-vp} to estimate the first three terms:
  \begin{align*}
    &a_h(\bfu - \pi_h \bfu, \bfv_h)
    +
    b_h(\bfu - \pi_h \bfu, q_h)
    +
    b_h(\bfv_h, p - \pi_h p)
    \\
    &\qquad\leqslant  C \bigl( \tn \bfu - \pi_h \bfu \tn \; \tn \bfv_h \tnast
    +\tn \bfu - \pi_h \bfu  \tn \; \tn q_h \tnast
    + \tn \bfv_h \tnast \; \tn p - \pi_h p \tn \bigr)
    \\
    &\qquad\leqslant
    C h \bigl( | \bfu |_{2,\Omega} + | p |_{1,\Omega} \bigr)
    \tn (\bfv_h, q_h )\tnast.
  \end{align*}
  Using the fact that $\Delta \pi_h \bfu = 0$ locally and applying the
  Cauchy--Schwarz inequality, we may estimate the remaining term
  $c_h = c_h(\bfu - \pi_h \bfu; p - \pi_h p, q_h)$ by
  \begin{align*}
    c_h
    &= \beta_0\sum_{F\in \partial_i \mesh} h_F ( \jump{p - \pi_h p}, \jump{q_h} )_F
    + \beta_1 \sum_{T \in \mesh} h_T^2 ( -\Delta \bfu + \nabla (p - \pi_h p_h), \nabla q_h)_T
    \\
    &\leqslant
    C \Bigl(
    \bigl(
    \sum_{T\in \mesh} h_T^2 \|\Delta \bfu \|_T^2 \bigr)^{1/2}
    +
    \bigl(
    \sum_{T\in \mesh} h_T^2 \|\nabla( p - \pi_h p )\|_T^2
    \bigr)^{1/2}
    \Bigr)
    \Bigl(
    \sum_{T\in \meshast} h_T^2 \| \nabla q_h \|_T^2
    \Bigr)^{1/2}
    \\
    &\quad +
    C \bigl(
    \sum_{F\in \partial_i \mesh} h_F \|\jump{ p - \pi_h p }\|_F^2
    \bigr)^{1/2}
    \Bigr)
    \Bigl(
    \sum_{F\in \partial_i \meshast} h_F \|\jump{ q_h}\|_F^2
    \Bigr)^{1/2}
    \\
    &\leqslant
    C h \bigl( | \bfu |_{2,\Omega} + | p |_{1,\Omega} \bigr)
    \tn (\bfv_h, q_h) \tnast,
  \end{align*}
  where we used the trace
  inequality~\eqref{p2:eq:trace-inequality} and inverse
  estimate~\eqref{p2:eq:inverse-estimates-for-triangles} for the last
  term to pass to $\| q \|_{\meshast} = \tn q_h \tnast$. Collecting
  all terms and applying the weak consistency estimate
  \eqref{p2:eq:weak-consistency} for $J_h$, we conclude that
    \begin{equation}
      \tn \pi_h \bfu  - \bfu_h \tn + \tn \pi_h p - p_h \tn
      \leqslant Ch (| \bfu |_{2,\Omega} + | p |_{1,\Omega}),
    \end{equation}
    since $\tn (\bfv_h, q_h ) \tnast = 1$.
\end{proof}

%------------------------------------------------------------------------------
\section{Condition number estimate}
\label{p2:sec:condition-number}

Following the approach of~\citet{Ern2006}, we now provide an estimate
for the condition number of the stiffness matrix associated with the
finite element formulation presented in
Section~\ref{p2:ssec:stokes-fd-fem}. In particular, the estimate shows
that the condition number is uniformly bounded by $C h^{-2}$
independently of the location of the boundary $\Gamma$ relative to the
background mesh $\widehat{\meshast}$.

First, we introduce some basic notation including the definition of
the condition number. Let $\{ \varphi_i \}_{i=1}^N$ be a basis for
some finite element space $\Wspace$. Then the expansion $v_h =
\sum_{i=1}^N V_i \varphi_i$ defines an isomorphism $\mcC: \Wspace
\rightarrow \R^N$ such that $\mcC v_h = V$, where $V = [V_1\dots
  V_N]^T$. We let $(V,W)_N = \sum_{i=1}^N V_i W_i$ denote the inner
product in $\R^N$ and $|V|_N^2 = (V, V)_N$ the corresponding Euclidean
norm.

We introduce the stiffness matrix $\mcA$ such that
\begin{equation}
  \label{p2:eq:def-stiffness-matrix}
  (\mcA V, W)_N
  = A_h( \bfv_h, q_h; \bfw_h, r_h) + J_h(\bfv_h, q_h; \bfw_h, r_h)
\end{equation}
for all $\bfv_h,\bfw_h \in V_h$ and all $q_h,r_h \in Q_h$ where
$V=\mcC (\bfv_h, q_h)$ and $W=\mcC (\bfw_h,r_h)$.  Since we consider
the Stokes problem for an enclosed flow with the velocity prescribed
on the entire boundary $\pO$, the solution is only determined up to a
constant pressure mode.  Consequently, the matrix $\mathcal{A}$ is
singular with kernel $\Kern(\mathcal{A}) = \spann\{
\mcC(\boldsymbol{0},1)\}$.  Throughout the remaining part of this
section, we therefore interpret $\mcA$ as the bijective linear mapping
between the $\widehat{\R}^N \to \widetilde{\R}^N$,
where $\widehat{\R}^N$ denotes the quotient space
$\widehat{\R}^N = \R^N / \Kern(\mcA)$ and
$\widetilde{\R}^N = \Image(\mcA) = \Kern(\mcA)^\bot$ the image
space (note that $\mcA$ is symmetric).  The condition number is
defined by
\begin{equation}\label{p2:eq:defkappa}
  \kappa(\mcA) = | \mcA |_N  | \mcA^{-1} |_N,
\end{equation}
with the operator norm
\begin{equation}
  | \mcA |_N = \sup_{V \in \widehat{\R}^N \setminus
    \boldsymbol{0}} \frac{|\mcA V|_N}{| V |_N }.
\end{equation}
Equivalently, the operator norm $| \mcA |$ may be defined by
\begin{equation}
  | \mcA |_N =
  \sup_{V \in \widehat{\R}^N
    \setminus
    \boldsymbol{0}}
  \sup_{W \in \widehat{\R}^N
    \setminus
    \boldsymbol{0}}
  \frac{(\mcA V, W)_N}{| V |_N | W |_N}.
\end{equation}

For a conforming, quasi-uniform mesh $\mesh$ with mesh size $h$ and a
finite element space $\mcV_h$ defined on $\mesh$, it is well known
that there are constants $c_{\mu} > 0$ and $C_{\mu} >$ only depending
on the uniformity parameters and the polynomial order of $\mcV_h$ such
that the following equivalence holds:
\begin{align}
  \label{p2:eq:VRineq}
  c_{\mu} h^{d/2} | V |_N \leqslant \| v_h \| \leqslant C_{\mu} h^{d/2} | V |_N  \quad
  \foralls v_h \in \mcV_{h}.
\end{align}
The following two lemmas are concerned with an inverse estimate and a
Poincar\'e inequality for the appropriate norms.
\begin{lem}
  There is a constant $C_I > 0$ such that
  \begin{alignat}{2}
    \label{p2:eq:L2energyinv}
    \tn \bfv_h \tnast &\leqslant C_I h^{-1} \| \bfv_h \|_{\Oast} &\quad
    &\foralls \bfv_h \in V_h, \\
    \label{p2:eq:L2energyinv-2}
    \tn (\bfv_h,q_h) \tnast &\leqslant C_I h^{-1} \| (\bfv_h,q_h)
    \|_{\Oast} &\quad &\foralls (\bfv_h,q_h)
    \in V_h \times Q_h.
  \end{alignat}
\end{lem}
\begin{proof}
  By definition $\tn \bfv_h \tnast^2 = \| \nabla \bfv_h \|_{\Oast}^2 +
  \| h^{-1/2} \bfv_h \|_{\pO}^2$. Hence, the
  inequality~\eqref{p2:eq:L2energyinv} follows from the applying the
  inverse estimate~\eqref{p2:eq:inverse-estimates-for-triangles} to the
  first term and the trace
  inequality~\eqref{p2:eq:trace-inequality-for-FD} and
  subsequently~\eqref{p2:eq:inverse-estimates-for-triangles} to the
  second term. The second estimate \eqref{p2:eq:L2energyinv-2} is a
  simple consequence noting that $1 \leqslant C h^{-1}\diam(\Omega)$.
\end{proof}

\begin{lem}
  (Poincar\'e inequality)
  There is a constant $C_P > 0$ such that
  \begin{align}
    \label{p2:eq:L2energy}
    \| \bfv_h \|_{\Oast} &\leqslant C_P \tn \bfv_h \tnast  \quad
    \foralls \bfv_h \in V_h.
  \end{align}
\end{lem}
\begin{proof}
  First we observe that $ \| \bfv_h \|_{\Oast} =
  \| \bfv_h \|_{\Omega} + \|  \bfv_h \|_{\Oast \setminus \Omega}
  \leqslant \| \bfv_h \|_{\Omega} + \|  \bfv_h \|_{\meshast_{\pO}}.
  $
  To estimate $\| \bfv_h \|_{\Omega}$, we apply a variant of the
  standard Poincar\'e inequality, valid for $\bfv \in
  [H^1(\Omega)]^d$~\citep{BrennerScott2008}:
  \begin{equation*}
    \| \bfv \|_{\Omega}^2 \leqslant C( \| \nabla  \bfv \|_{\Omega}^2 +
    \| \bfv \|_{\pO}^2).
  \end{equation*}
  Since $\| \bfv \|_{\pO} \leqslant C \| h^{-1/2} \bfv \|_{\pO}$, we
  conclude that $\| \bfv_h \|_{\Omega} \leqslant C \tn \bfv_h \tnast$.
  Using this estimate and the definition of $\tn \bfv_h \tnast$, a bound
  for the remaining term $\| \bfv_h \|_{\meshast_{\pO}}$ can be
  obtained as in the proof for~\eqref{p2:eq:p1-fd-norm-est} (noting
  that $h_F^3 \leqslant C h_F$):
  \begin{align*}
    \| \bfv_h \|_{\meshast_{\pO}}^2
    &\leqslant
    C (\| \bfv_h \|_{\Omega}^2
    + \sum_{F\in\Fast} h^3_F([\nablan \bfv_h],[\nablan
      \bfv_h])_{F})
    \\
    &\leqslant
    C (\|h^{-1/2} \bfv \|^2_{\pO} + \| \nabla\bfv_h \|_{\Omega}^2
    + \sum_{F\in\Fast} h_F([\nablan \bfv_h],[\nablan
      \bfv_h])_{F}).
  \end{align*}
  The last two terms are bounded by $\|\nabla \bfv_h \|_{\Oast}$
 by~\eqref{p2:eq:u-fd-norm-est}, thus
  yielding~\eqref{p2:eq:L2energy}.
\end{proof}

Finally, we state the
continuity of the overall form $A_h + J_h$
with respect to the norm $\tn \cdot \tnast$:
\begin{lem}
There exists a constant $C_A$ such that for all $(\bfv_h,
q_h),\,(\bfw_h, r_h) \in V_h \times Q_h$
\begin{equation}
  A_h(\bfv_h,q_h;\bfw_h,r_h) +
  J_h(\bfv_h,q_h;\bfw_h,r_h)
  \leqslant C_A \tn (\bfv_h, q_h ) \tnast \, \tn (\bfw_h, r_h) \tnast.
  \label{p2:eq:Ah_continuity}
\end{equation}
\end{lem}
\begin{proof}
  Because of the continuity estimates~\eqref{p2:eq:a_h-stable-2}
  and~\eqref{p2:eq:b-cont-wrt-norm-2}, it only remains to estimate the
  contribution $c_h(q_h, r_h)$, which follows the same lines as in the
  proof of Theorem~\ref{thm:a-priori-estimate}.
\end{proof}

We are now in the position to state the main result of this section.
\begin{thm}
  The condition number of the stiffness matrix $\mcA$ associated with
  the Nitsche fictitious domain method \eqref{p2:eq:stokes-fd} satisfies
  the estimate
  \begin{equation}
    \kappa(\mcA) \leqslant C h^{-2}.
  \end{equation}
\end{thm}
\begin{proof}
  Recalling the definition of the condition number in
  \eqref{p2:eq:defkappa}, the proof consists of deriving estimates for
  $|\mcA|_N$ and $|\mcA^{-1}|_N$.
  By definition, for all $V, W \in \widehat{\R}^N \setminus
  \{\mathbf{0}\}$,
  \begin{align*}
    (\mcA V, W)_N
    &= A_h(\bfv_h, q_h; \bfw_h, r_h) + J_h(\bfv_h, q_h; \bfw_h, r_h) \\
    &\leqslant C_A \tn (\bfv_h, q_h) \tnast \cdot \tn (\bfw_h, r_h) \tnast \\
    & \leqslant C_A C_I^{2} h^{-2}
    ||(\bfv_h, q_h)||_{\Oast} \cdot ||(\bfw_h, r_h)||_{\Oast}\\
    & \leqslant C_A C_I^{2} C_{\mu}^2 h^{d-2} |V|_N |W|_N ,
  \end{align*}
  where the inequalities follow from the continuity of $A_h + J_h$,
  the inverse estimate~\eqref{p2:eq:L2energyinv-2}, and
  finally~\eqref{p2:eq:VRineq}. Thus
  \begin{equation}
    |\mcA|_{N} \leqslant C_A C_I^{2} C_{\mu}^2 h^{d-2} .
    \label{p2:eq:Nnorm-A}
  \end{equation}
  Similarly, for all $V \in \widehat{\R}^N \setminus \{\mathbf{0}\}$,
  there exists a $W$ such that
  \begin{align*}
    (\mcA V, W)_N
    &= A_h(\bfv_h, q_h; \bfw_h, r_h) + J_h(\bfv_h, q_h; \bfw_h, r_h) \\
    &\geqslant c_{A} \tn (\bfv_h, q_h) \tnast \cdot \tn (\bfw_h, r_h) \tnast \\
    &\geqslant C_{P}^{-2} c_{A} || (v_h, q_h) ||_{\Oast}
    \cdot || (w_h, r_h) ||_{\Oast} \\
    & \geqslant C_{P}^{-2} c_{A} c_{\mu}^2 h^d |V|_N |W|_N ,
  \end{align*}
  where the inequalities follow from the inf-sup
  estimate~\eqref{p2:eq:stability}, the Poincar\'e
  inequality~\eqref{p2:eq:L2energy}, and
  finally~\eqref{p2:eq:VRineq}. Moreover,
  \begin{equation*}
    | \mcA V |_N
    = \sup_{Y} (\mcA V, Y) | Y |_N^{-1}
    \geqslant (\mcA V, W) | W |_{N}^{-1}
    \geqslant C |V|_N,
  \end{equation*}
  where $C = C_{P}^{-2} c_{A} c_{\mu}^2 h^d$. Letting $V = \mcA^{-1}
  Y$, which is allowed since $\mcA^{-1}$ is indeed invertible on the
  reduced space, and rearranging the inequality, we obtain
  $|\mcA^{-1} Y|_N \leqslant C^{-1} |Y|_N$ for all $Y$, and so
  \begin{equation}
    \label{p2:eq:Nnorm-Ainv}
    |\mcA^{-1}|_N \leqslant C_{P}^{2} c_{A}^{-1} c_{\mu}^{-2} h^{-d}.
  \end{equation}
  Combining~\eqref{p2:eq:Nnorm-A} and~\eqref{p2:eq:Nnorm-Ainv}, we obtain
  the desired estimate
  \begin{equation*}
    \kappa(\mcA) = |\mcA|_N  \, |\mcA^{-1}|_N
    \leqslant C_I^{2} C_{P}^{2} \frac{C_A}{c_A} \frac{C_{\mu}^2}{c_{\mu}^2} h^{-2} .
  \end{equation*}
\end{proof}

%------------------------------------------------------------------------------
\section{Numerical examples}
\label{p2:sec:num-examples}

\subsection{Software for fictitious domain variational formulations}

The assembly of finite element tensors corresponding to standard
variational formulations on conforming, simplicial meshes, such
as~\eqref{p2:eq:stabilized-stokes-form}, involves integration over
elements and possibly, interior and exterior facets. In contrast, the
assembly of variational forms defined over fictitious domains, such
as~\eqref{p2:eq:fict-domain-stokes-A}, \eqref{p2:eq:fict-domain-stokes-J}
and~\eqref{p2:eq:stokes-fd-rhs}, additionally requires integration over
cut elements and cut facets. These mesh entities are of polyhedral,
but otherwise arbitrary, shape. As a result, the assembly process is
highly non-trivial in practice and requires additional geometry
related preprocessing, which is challenging in particular for
three-dimensional meshes.

As part of this work, the technology required for the automated
assembly of general variational forms defined over fictitious domains
has been implemented as part of the software library
\rm{DOLFIN-OLM}. This library builds on the core components of the
FEniCS Project~\citep{LoggMardalEtAl2011,Logg2007}, in particular
DOLFIN~\citep{LoggWells2010a}, and the computational geometry
libraries \rm{CGAL}~\citep{cgal} and
\rm{GTS}~\citep{gts}. \rm{DOLFIN-OLM} is open source and freely
available from \url{http://launchpad.net/dolfin-olm}.

There are two main challenges involved in the implementation: the
computational geometry and the integration of finite element
variational forms on cut cells and facets. The former involves
establishing a sufficient topological and geometric description of the
fictitious domain for the subsequent assembly process. To this end,
\rm{DOLFIN-OLM} provides functionality for finding and computing the
intersections of triangulated surfaces with arbitrary simplicial
background meshes in three spatial dimensions; this functionality
relies on the computational geometry libraries \rm{CGAL} and
\rm{GTS}. These features generate topological and geometric
descriptions of the cut elements and facets. Based on this
information, quadrature rules for the integration of fields defined
over these geometrical entities are produced. The computational
geometry aspect of this work extends, but shares many of the features
of, the previous work~\citep{Massing2012a}, and is described in more
detail in the aforementioned reference.

Further, by extending some of the core components of the FEniCS
Project, in particular FFC~\citep{KirbyLogg2006,LoggOelgaardEtAl2011a}
and UFC~\citep{AlnaesLoggEtAl2012a}, this work also provides a finite
element form compiler for variational forms defined over fictitious
domains. Given a high-level description of the variational
formulation, low-level C++ code can be automatically generated for the
evaluation of the cut element, cut facet and surface integrals, in
addition to the evaluation of integrals over the standard (non-cut)
mesh entities. The generated code takes as input appropriate
quadrature points and weights for each cut element or facet; these are
precisely those provided by the \rm{DOLFIN-OLM} library.

As a result, one may specify variational forms defined over finite
element spaces on fictitious domains in high-level \rm{UFL} notation
\citep{Alnaes2011a}, define the background mesh $\widehat{\meshast}$
and give a description of the surface $\Gamma$, and then invoke the
functionality provided by the \rm{DOLFIN-OLM} library to automatically
assemble the corresponding stiffness matrix. In particular, the
numerical experiments presented below, corresponding to the
variational formulation defined by~\eqref{p2:eq:fict-domain-stokes-A},
\eqref{p2:eq:fict-domain-stokes-J} and~\eqref{p2:eq:stokes-fd-rhs}, have
been carried out using this technology.

\subsection{Convergence rates}

To corroborate the theoretical error
estimate~\eqref{p2:eq:a-priori-estimate} by numerical results, we
consider a basic test case with a manufactured exact solution and
compute the errors in the velocity and the pressure approximations on
sequences of refined meshes. To this end, let $\Omega = [0, 1]^3$ with
$\Gamma = \partial \Omega$. To examine the convergence of the Nitsche
fictitious domain method, we apply the method of manufactured
solutions. Let
\begin{equation*}
  \bfu(x,y,z) = (y (1-y) z(1-z), 0, 0), \quad p(x,y,z) = 0.5 - x.
\end{equation*}
The right-hand side $\bff$ is defined accordingly and the
corresponding Dirichlet boundary conditions are applied via the
Nitsche method on the entire boundary $\pO$ such that $\bfu$ and $p$
solve the Stokes problem~\eqref{p2:eq:strongform}.

Let $\delta = 0.01$ be a perturbation factor. We define three
different families of mesh configurations, each parametrized over $N$
with $h = 1/N$, for the background domain $\Oast$:
\begin{enumerate}
  \item[(A)]
    $\Oast = [-h \delta, 1 + h \delta]^3$, divided into $N^3$ subcubes;
  \item[(B)]
    $\Oast = [ -h/3, 1 + h/3]^3$, divided into $N^3$ subcubes;
  \item[(C)]
    $\Oast = [-h(1-\delta), 1 + h(1-\delta)]^3$, divided into
    $(N+2)^3$ subcubes.
\end{enumerate}
The final meshes result from tessellating each subcube into $6$
tetrahedra. For the scenario (A), the background mesh is almost
entirely covered by the domain $\Omega$; while scenario (C) represents
the other extreme: the outermost layer of tetrahedra is only barely
intersected by $\Omega$. Scenario (B) illustrates a middle ground.

For the case $V_h \times \Pone$, we take $\beta_1 = 0.2$, $\beta_2 =
1.0, \beta_3 = 0.05$ and $\gamma = 10$ as the stabilization parameters
involved in~\eqref{p2:eq:stokes-fd}; while for $V_h \times \Pzero$, we
take $\beta_0 = 0.25$, $\beta_2 = 0.1$ and $\gamma =
10$. To solve the resulting systems of equations, we apply a
transpose-free quasi-minimal residual (TFQMR) solver with an algebraic
multigrid preconditioner. The constant pressure mode is filtered out
in the iterative solver. We observed that the iterative solvers
converged in between $6$ and $25$ iterations. The $[H^1(\Oast)]^d$
error of the velocity approximation and the $L^2(\Oast)$ error of the
pressure approximation were computed, using the natural extensions of
the exact solutions to $\Oast$, for each mesh configuration and a
series of mesh sizes.

\begin{figure}
  \begin{center}
    \includegraphics[width=0.83\textwidth]{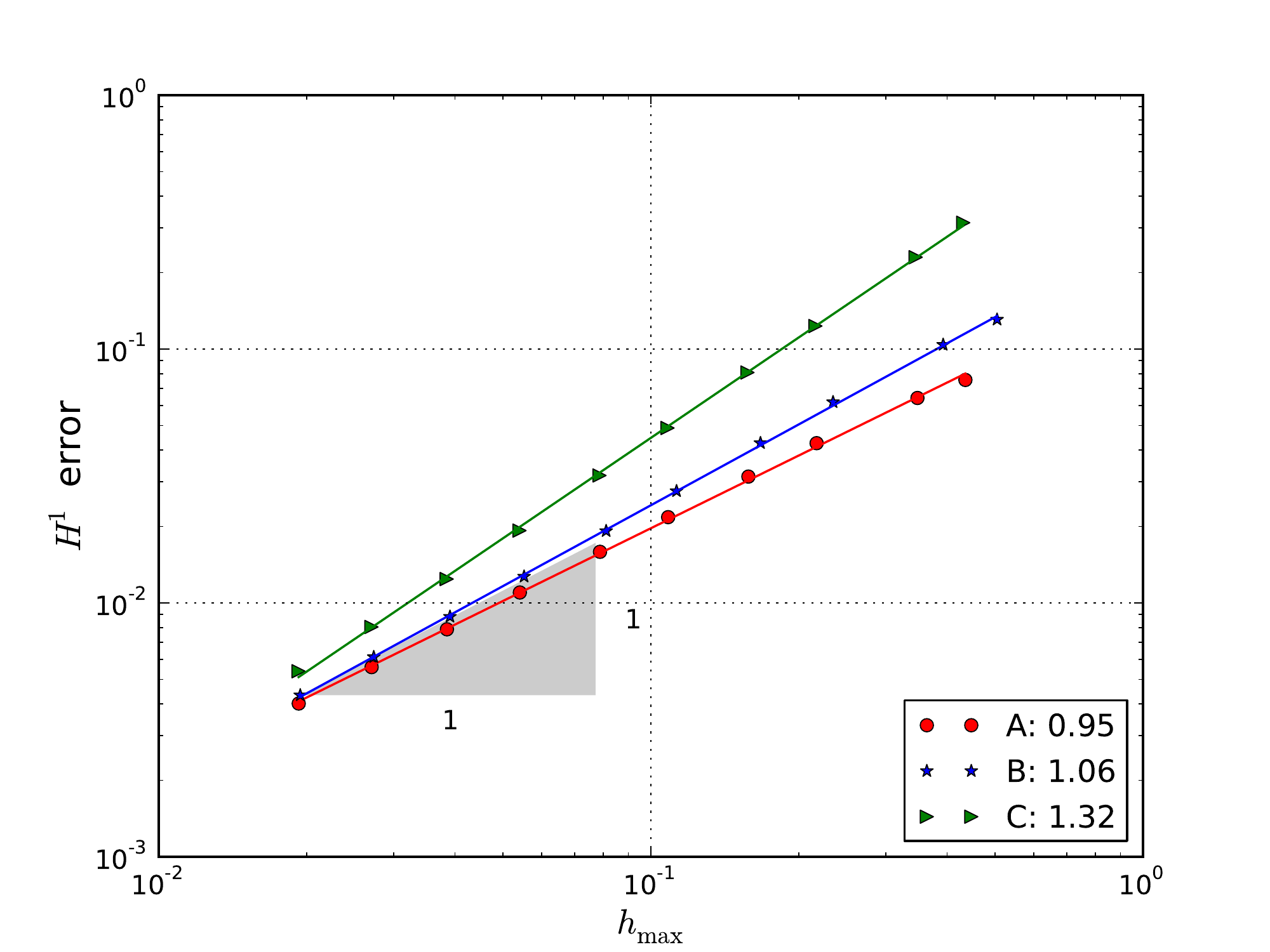}
    \\
    \includegraphics[width=0.83\textwidth]{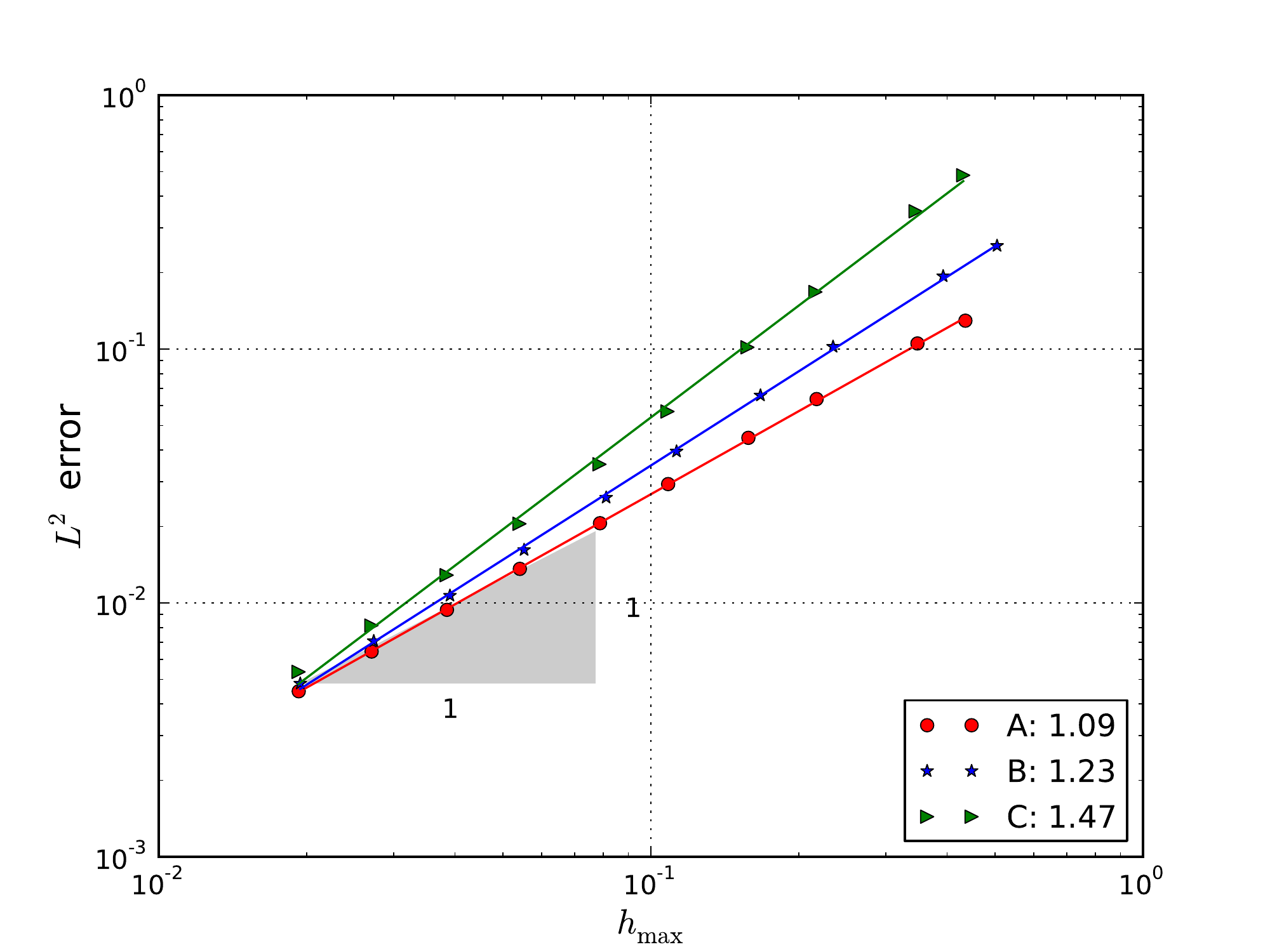}
    \caption{The case $V_h \times \Pzero$: errors for the three
      different mesh configurations (A), (B) and (C) versus maximal
      element diameter $h_{\max}$. The legend gives the fitted slope
      for each configuration. Top: $H^1$-error $||\bfu - \bfu_h||_{1,
        \Oast}$ for the velocity. Bottom: $L^2$-error $||p -
      p||_{\Oast}$ for the pressure.}
      \label{fig:convergence:p1xp0}
  \end{center}
\end{figure}
\begin{figure}
  \begin{center}
    \includegraphics[width=0.83\textwidth]{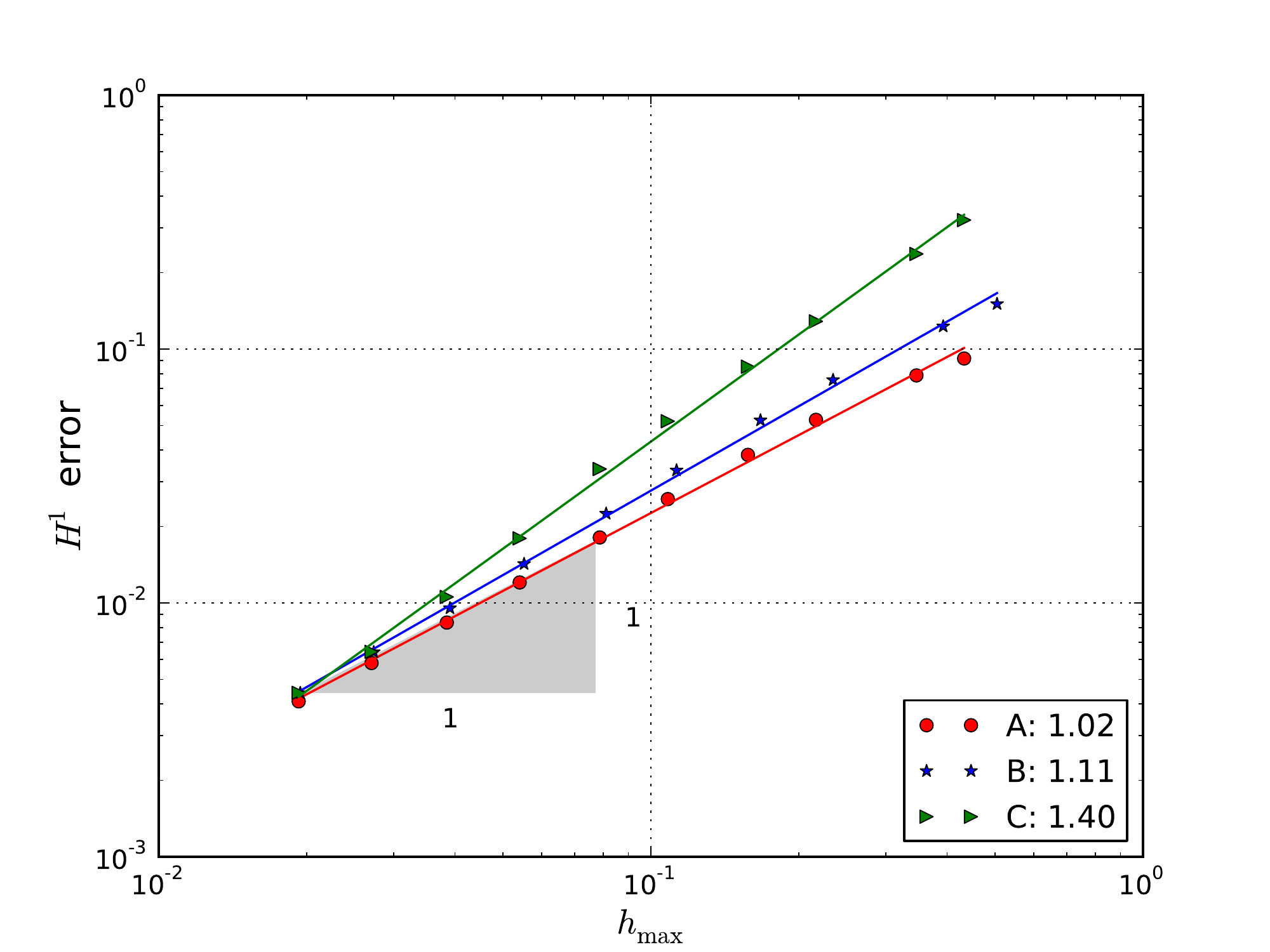} \\
    \includegraphics[width=0.83\textwidth]{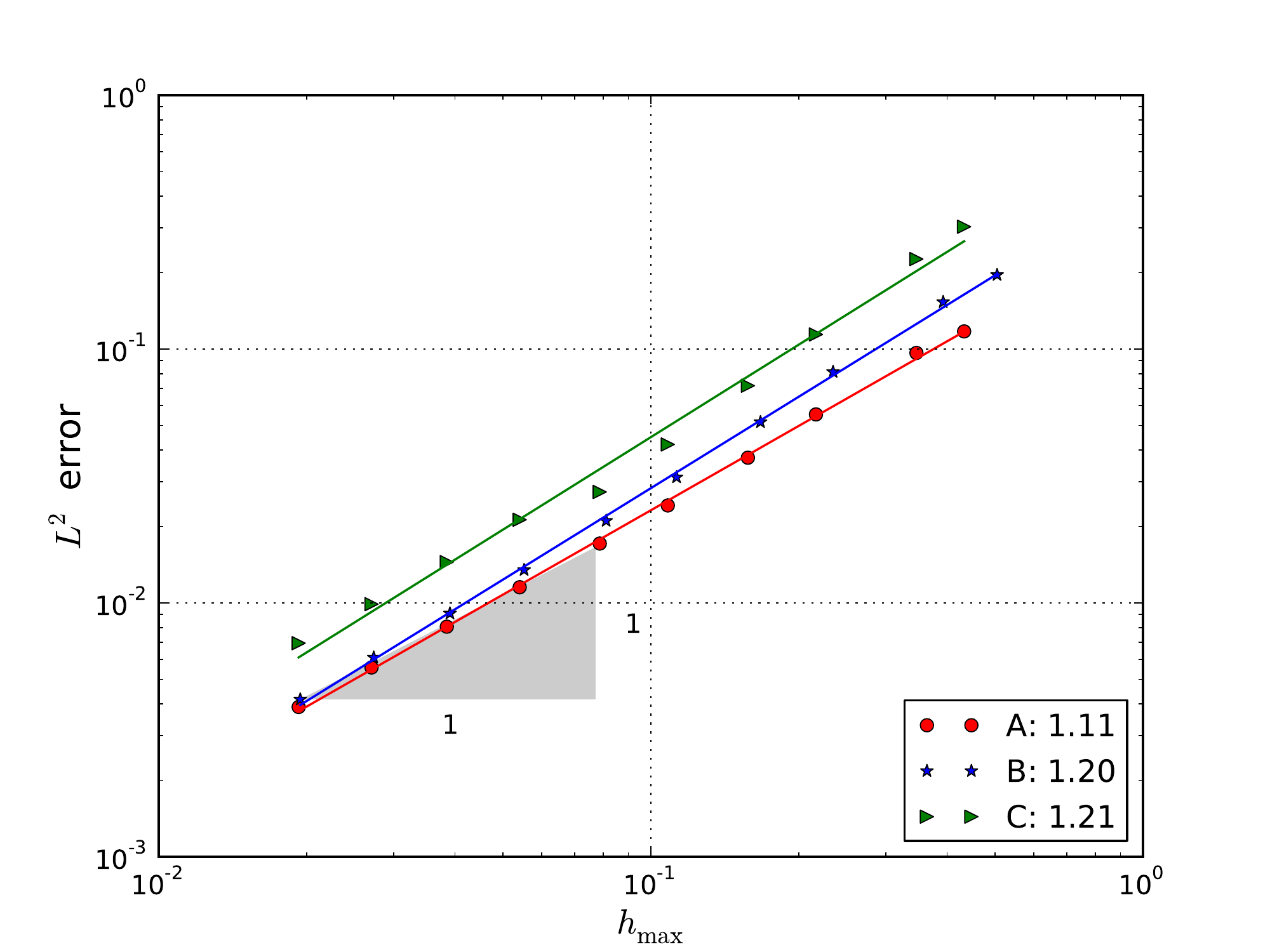}
    \caption{The case $V_h \times \Pone$: errors for the three
      different mesh configurations (A), (B) and (C) versus maximal
      element diameter $h_{\max}$. The legend gives the fitted slope
      for each configuration. Top: $H^1$-error $||\bfu -
      \bfu_h||_{1, \Oast}$ for the velocity.  Bottom: $L^2$-error $||p
      - p||_{\Oast}$ for the pressure.}
    \label{fig:convergence:p1xp1}
  \end{center}
\end{figure}

The resulting errors are plotted in Figure~\ref{fig:convergence:p1xp0}
and Figure~\ref{fig:convergence:p1xp1} for $V_h \times \Pzero$ and
$V_h \times \Pone$, respectively. Theorem~\ref{thm:a-priori-estimate}
predicts first order convergence for the $H^1$-norm of the velocity
error and the $L^2$-norm of the pressure error. These orders are also
obtained in the numerical experiments: both for $V_h \times \Pzero$
and $V_h \times \Pone$ and each of the three different scenarios, the
errors monotonically decrease and seem to converge towards zero by (at
least) first order.

We note that the results for the different scenarios illustrate that
the positioning of background mesh does affect the magnitude of the
errors to some extent. For the scenario (A), the convergence rates for
both the velocity and the pressure seem fairly uniform over the range
of mesh sizes considered. We observe the same for the scenario (B),
though the errors and rates are a little higher. For the scenario (C),
the convergence rates for the $L^2$ norm of the pressure are somewhat
less uniform for the case $V_h \times \Pone$, and the errors and rates
are again higher for both pairs of finite element spaces. As a
consequence, we remark that for a series of background meshes where
the location of the surface varies significantly with respect to the
mesh configuration, non-monotone decrease of the errors may be
observed. We also note that superconvergence is observed and is most
clearly pronounced in scenario~(C). This is related to the definition
of the norms $\|\cdot\|_{1,\Oast}$ and $\|\cdot\|_{\Oast}$ which
extend to the entire fictitious domain~$\Oast$. In scenario~(C), the
fictitious domain $\Oast$ extends a distance~$h$ from the boundary of
the computational domain~$\Omega$. The volume of $\Oast$ will thus
decrease in size during mesh refinement and contribute to the observed
rates of superconvergence.

\subsection{Influence of the boundary position on the condition
  number}
\label{p2:ssec:condition-number-tests}

Next, we consider a numerical example to demonstrate that the
condition number of the matrix $\mcA$ corresponding to the stabilized
fictitious domain bilinear form, as defined
by~\eqref{p2:eq:def-stiffness-matrix}, is bounded and that the bound
is independent of the boundary position relative to the background
mesh.

We consider the domain $\Oast = [-1,1]^3$ tessellated by uniformly
dividing the domain into $10^3$ cubes, with each cube subdivided into
6 tetrahedra. The domain $\Omega = \Omega(l)$ is defined by
$[-l,l]^3$, where we have in mind $l$ ranging from $0.9$ to
$1.0$. Note that when $l$ is close to $1.0$, almost the entire
background mesh is included in the computational domain. On the other
hand, as $l$ approaches $0.9$, some of the outermost elements of the
background mesh will only barely intersect $\Omega$. So, as $l$ varies
between $1.0$ and $0.9$, the smallest ratio $r$ of $|T \cap \Omega|$
to $|T|$ for the elements $T$ in the outermost layer varies between
$1.0$ and $0.0$. For each $l$, we compute the condition number of the
corresponding matrix $\mcA$, letting $\beta_0 = \beta_1 = 0.1$,
$\gamma = 10$, and varying $\beta_2 = \beta_3 = \beta$.  The condition
number was computed as the ratio of the absolute value of the largest
(in modulus) eigenvalue and the smallest (in modulus) nonzero
eigenvalue of the symmetric matrix $\mcA$.
\begin{table}
\centering
\begin{tabular}{c|rrrr}
\toprule
$\beta$/$l$ &  $0.990$ &  $0.950$ &  $0.910$ &  $0.901$  \\
\midrule
$0.0$ &  $386$ & $1544$ & $176467$ & $174485837$  \\
$0.001$ &  $378$ & $1064$ & $4037$ & $4643$  \\
$0.01$ &  $360$ & $607$ & $1048$ & $1161$  \\
$0.025$ &  $395$ & $580$ & $857$ & $928$  \\
$0.05$ &  $486$ & $670$ & $928$ & $994$  \\
$0.1$ &  $689$ & $915$ & $1224$ & $1303$  \\
$1.0$ &  $4435$ & $5534$ & $6931$ & $7291$  \\
$10.0$ &  $51986$ & $62711$ & $75764$ & $79066$  \\
\bottomrule
\end{tabular}
\caption{Scaled condition numbers for $V_h \times \Pone$ with varying
  ghost-penalty stabilization parameters $\beta = \beta_2 = \beta_3$
  (each row corresponds to one $\beta$), a varying domain $\Omega =
  [-l, l]^3$ and fixed background domain $\Oast = [-1, 1]^3$.}
\label{tab:condition:p1xp1}
\end{table}
\begin{table}
\centering
\begin{tabular}{c|rrrr}
\toprule
$\beta$/$l$ & $0.990$ & $0.950$ & $0.910$ & $0.901$ \\
\midrule
$0.0$ &  $1175$ & $1649$ & $5777$ & $4191056$  \\
$0.001$ &  $1178$ & $1653$ & $6650$ & $2481$  \\
$0.01$ &  $1229$ & $1707$ & $2373$ & $2533$  \\
$0.025$ &  $1431$ & $1952$ & $2625$ & $2771$  \\
$0.05$ &  $1859$ & $2523$ & $3381$ & $3565$  \\
$0.1$ &  $2803$ & $3828$ & $5180$ & $5487$  \\
$1.0$ &  $24313$ & $33152$ & $44964$ & $47954$  \\
$10.0$ &  $350160$ & $447888$ & $575179$ & $607977$  \\
\bottomrule
\end{tabular}
\caption{Scaled condition numbers for $V_h \times \Pzero$ with varying
  ghost-penalty stabilization parameter $\beta = \beta_2$
  (each row corresponds to one $\beta$), a varying domain $\Omega =
  [-l, l]^3$ and fixed background domain $\Oast = [-1, 1]^3$.}
\label{tab:condition:p1xp0}
\end{table}
The resulting condition numbers, scaled by the square mesh size $h^2
\approx 0.35^2$, for a series of $\beta$ and $l = 0.99, 0.95, 0.91,
0.901$ are given in Table~\ref{tab:condition:p1xp1} and
Figure~\ref{p2:fig:scaled_condition_number} for $V_h \times
\Pone$ and in Table~\ref{tab:condition:p1xp0} for $V_h \times
\Pzero$. First, consider the case $V_h \times \Pone$. For $\beta =
0.0$, the scaled condition number is low ($386$) when $l = 0.99$; that
is, when the ratio $r$ is almost $1$. However, the scaled condition
number increases dramatically as $l$, and hence the ratio~$r$ is
reduced. Thus, if no ghost-penalty terms are included, the scaled
condition number seems unbounded as $l$ tends to $0.9$. On the other
hand, in the cases where $\beta$ is positive, the scaled condition
number only grows moderately as the ratio is significantly reduced and
seems bounded. We note however that the condition number grows with
the penalty parameter $\beta$ for $\beta > 0.025$. Finally, similar
observations apply in the case $V_h \times \Pzero$
(Table~\ref{tab:condition:p1xp0}).
\begin{figure}
  \begin{center}
    \includegraphics[width=0.81\textwidth]{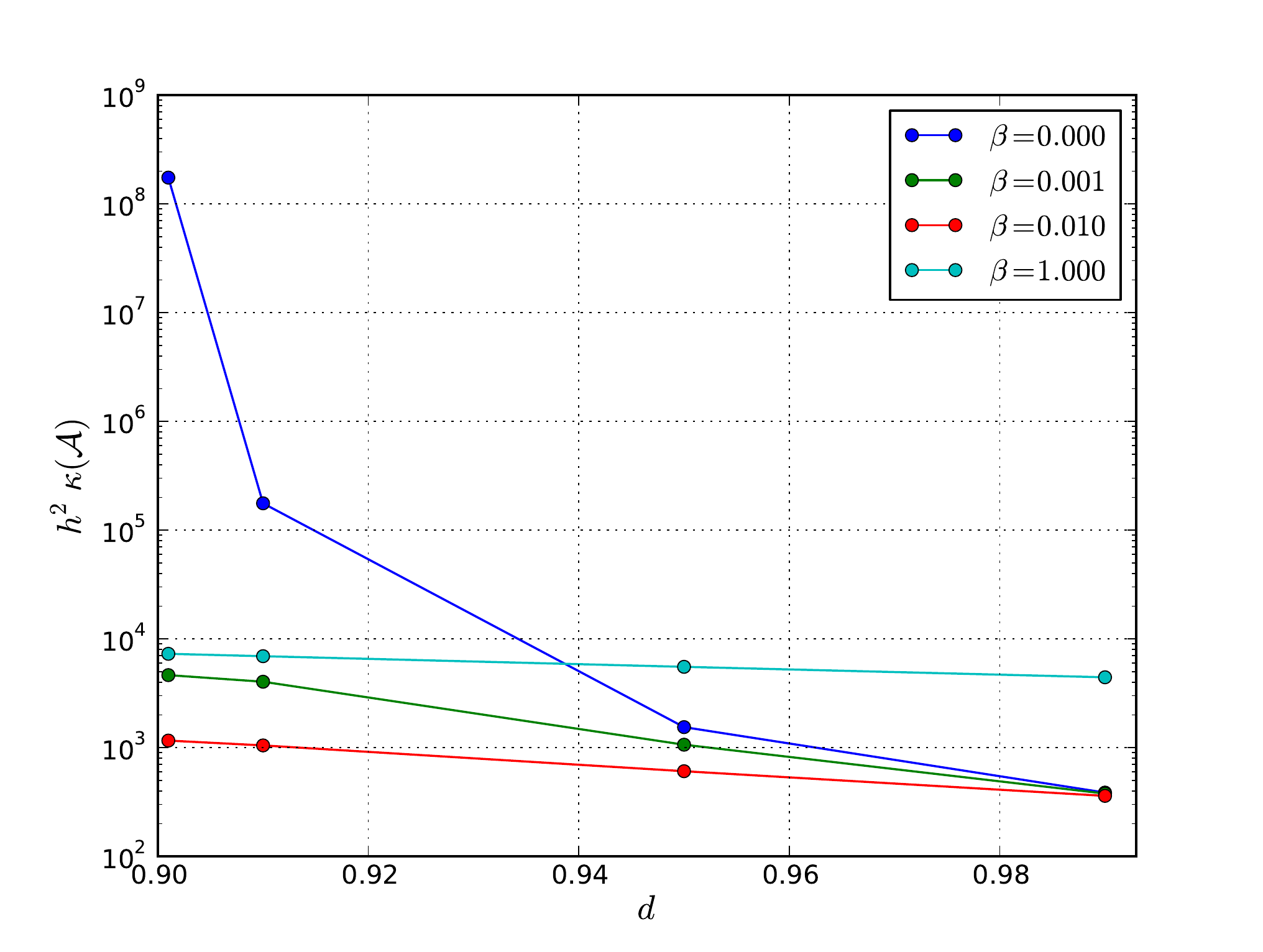}
    \caption{Semilogarithmic plot of the scaled condition number for $V_h
      \times \Pone$ with varying ghost-penalty stabilization
    parameters $\beta = \beta_2 = \beta_3$.}
      \label{p2:fig:scaled_condition_number}
  \end{center}
\end{figure}

\subsection{Stokes flow in a complex geometry}
\label{p2:ssec:flow-in-complex-geometry}
We conclude the section with an example of Stokes flow in a
computational domain where the boundary is described by a complex
surface geometry.  The geometry is taken from a part of an arterial
network known as the Circle of Willis which is located close to the
human brain.  It is known that the network is prone to develop
aneurysms and therefore the computer-assisted study of the blood flow
in the Circle of Willis has been a recent subject of interest, see for
instance \citet{Steinman2003,IsaksenBazilevsKvamsdalEtAl2008,Valen-Sendstad2011}.
However, the purpose of this example is not to perform a realistic
study of the blood flow dynamics. Rather, we would like to demonstrate
the principal applicability of the developed method to simulation
scenarios where complex three-dimensional geometries are involved.
The extension of the work to numerically solve the time-dependent
Navier--Stokes equations in a biomedical relevant regime is the
subject of future research.

The blood vessel geometry is embedded in a structured background mesh
as illustrated in Figure~\ref{fig:aneurysm-box-mesh}.
As before, the velocity is prescribed on the entire boundary $\Gamma$
where we set $\bfu = 0$ on the arterial walls and $\bfu = 1200\,
\mathrm{mm/s}$ on the inlet boundary. The two outflow velocities were
set in such a way that total flux was balanced.

The pressure  and velocity approximation as computed on the fictitious
domain mesh $\meshast$ are shown in Figure~\ref{fig:aneurysm-box-mesh}
and \ref{fig:aneurysm-velocity}, respectively.  Although the
fictitious domain mesh $\meshast$ provides only a coarse resolution of
the aneurysm geometry, the values of the velocity approximation
clearly conforms to the required boundary values on the actual surface
geometry.

\begin{figure}
  \begin{center}
    \includegraphics[width=0.61\textwidth]{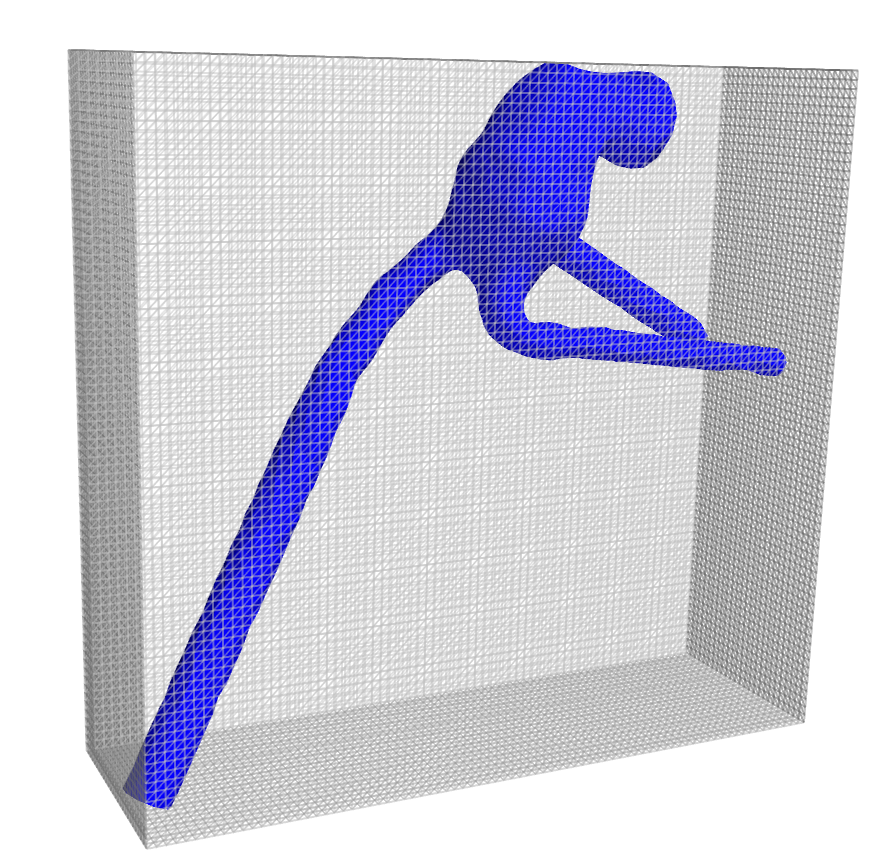}
    \\
    \includegraphics[width=0.81\textwidth]{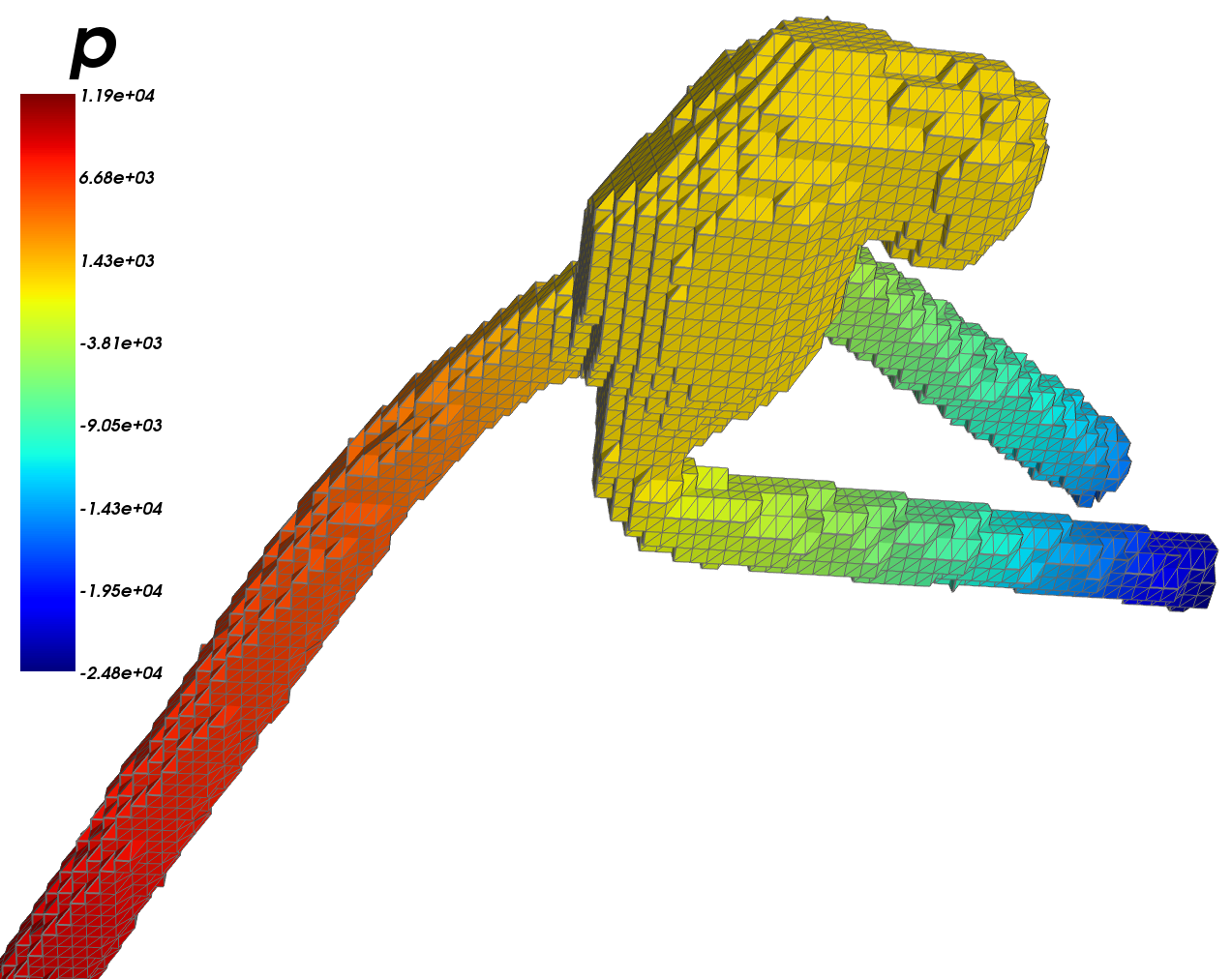}
    \caption{Stokes flow in an aneurysm. Top: Aneurysm surface
      embedded in the structured background mesh
      $\widehat{\mesh}^{\ast}$. Bottom: Fictitious domain $\Oast$ and
    corresponding mesh $\meshast$ with pressure approximation.}
    \label{fig:aneurysm-box-mesh}
  \end{center}
\end{figure}
\begin{figure}
  \begin{center}
    \includegraphics[width=0.9\textwidth]{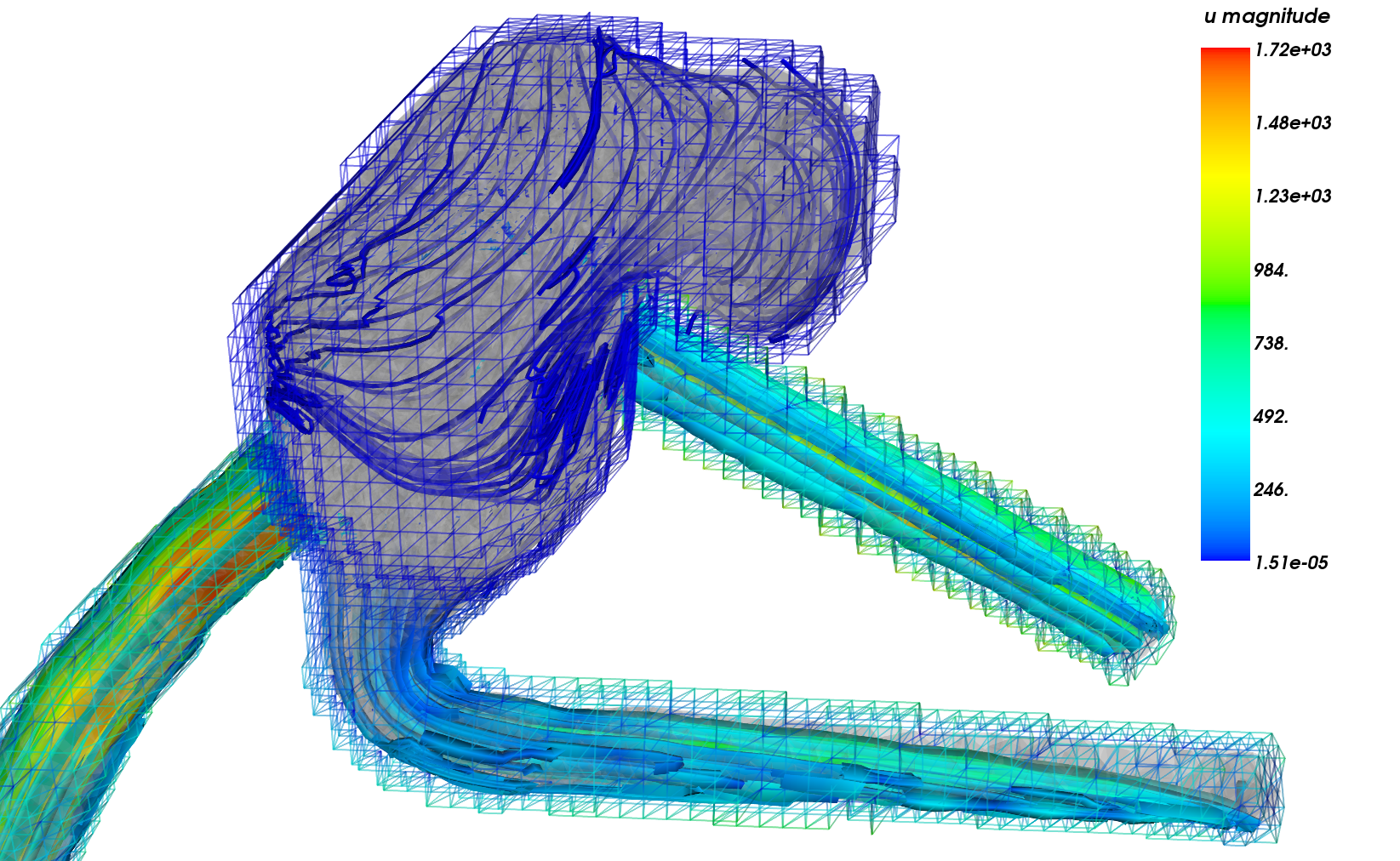}
    \\
    \includegraphics[width=0.9\textwidth]{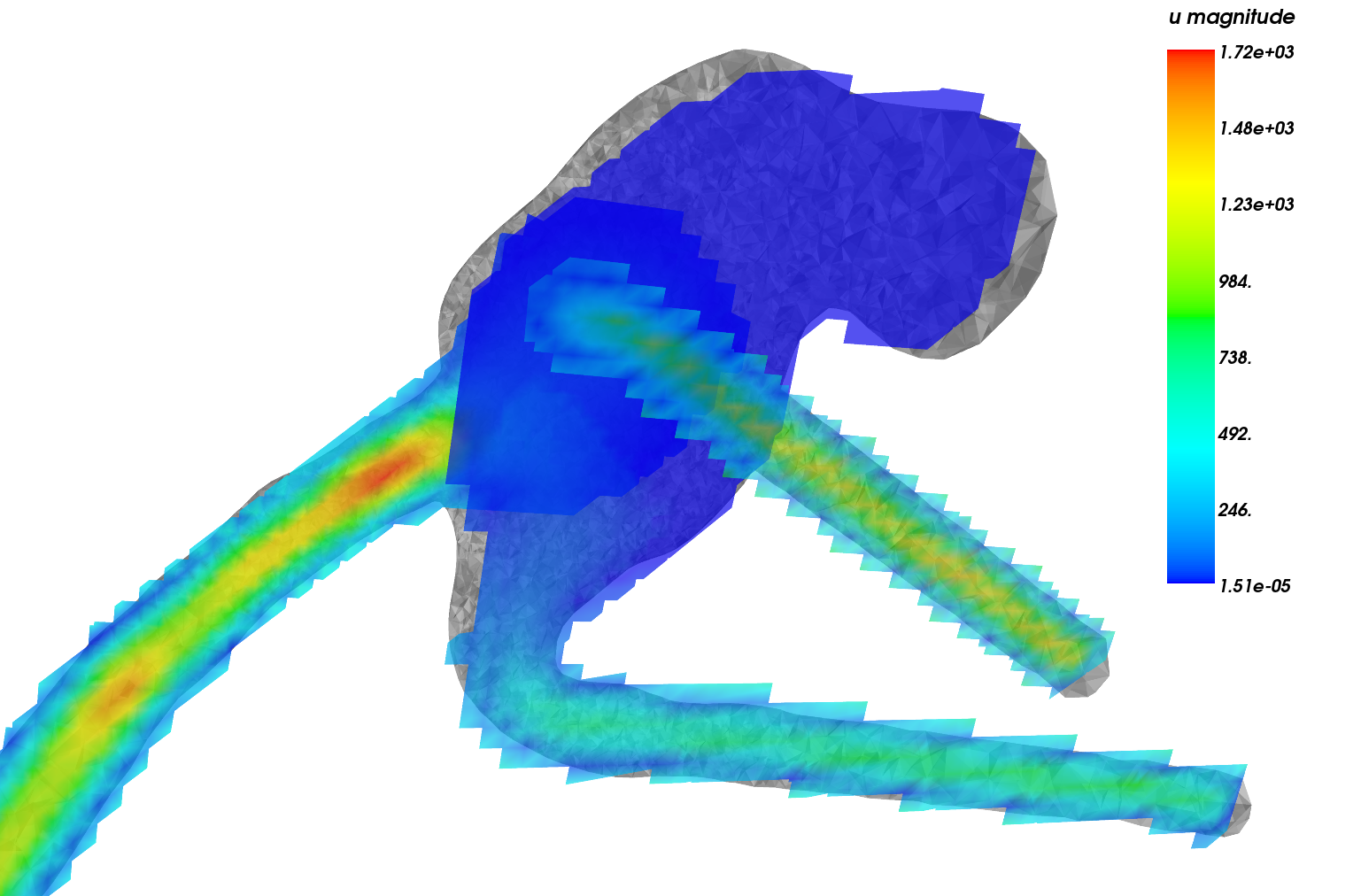}
    \caption{Velocity approximation on $\meshast$. Top: Original
    aneurysm embedded in the background mesh shown with velocity
  streamlines. Bottom: Three cross-section through the aneurysm
showing the magnitude of the velocity in each vessel section.
Despite the coarse approximation of the aneurysm geometry
by the fictitious domain mesh, the boundary values of the velocity
respect the underlying surface geometry.}
    \label{fig:aneurysm-velocity}
  \end{center}
\end{figure}

%------------------------------------------------------------------------------
\section{Conclusions}
\label{p2:sec:conclusion}

We have presented a stabilized finite element method for the solution
of the Stokes problem on fictitious domains and proved optimal order
convergence. The theoretical convergence rates have been verified
numerically. We have also proved that the condition number of the
stiffness matrix remains bounded, independently of the position of the
fictitious boundary relative to the background mesh.

While we have here restricted our attention to the static Stokes model
problem, the main motivation for the methodology and implementation
presented in this paper is for the treatment of the time-dependent
Navier--Stokes equations and, ultimately, fluid--structure interaction
on complex and evolving geometries. We address this issue in future
work.

%------------------------------------------------------------------------------
\section*{Acknowledgements}

The authors wish to thank Sebastian Warmbrunn for providing the
surface geometry used in
Section~\ref{p2:ssec:flow-in-complex-geometry} and Kent-Andre Mardal
for insightful discussion on preconditioning.  This work is supported
by an Outstanding Young Investigator grant from the Research Council
of Norway, NFR 180450. This work is also supported by a Center of
Excellence grant from the Research Council of Norway to the Center for
Biomedical Computing at Simula Research Laboratory.

\bibliographystyle{plainnat}
\bibliography{bibliography}

\end{document}